\documentclass[11pt,reqno,oneside]{amsart}

\usepackage{graphicx}
\usepackage{amssymb}
\usepackage{epstopdf}
\usepackage{cite}
\usepackage[a4paper, total={6in, 9in}]{geometry}
\usepackage{amsmath,amsfonts,amsthm,mathrsfs,amssymb,cite,dsfont}
\usepackage[usenames]{color}
\usepackage{bm}
\usepackage{subfigure}

\usepackage{algorithm}
\usepackage{algpseudocode}
\algrenewcommand\algorithmicrequire{\textbf{Input:}}
\algrenewcommand\algorithmicensure{\textbf{Output:}}
\usepackage{cases}
\usepackage{multirow}
\makeatletter
\def\BState{\State\hskip-\ALG@thistlm}
\makeatother

\usepackage{amsmath,amsfonts,amsthm,mathrsfs,amssymb,cite,mathrsfs}
\usepackage[usenames]{color}

\newtheorem{thm}{Theorem}[section]

\newtheorem{lem}{Lemma}[section]

\theoremstyle{definition}

\theoremstyle{remark}

\newtheorem{rem}{Remark}[section]
\numberwithin{equation}{section}
\numberwithin{equation}{section}

\newcounter{saveeqn}

\newcommand{\vect}[1]{\boldsymbol{#1}}

\newcommand{\la}{\langle}
\newcommand{\ra}{\rangle}
\newcommand{\bZ}{\mathbf{z}}
\newcommand{\eqnref}[1]{(\ref {#1})}

\title[Mathematical theory of magnetic anomaly detections]{Identifying variations of magnetic anomalies using geomagnetic monitoring}

\author{Youjun Deng}
\address{School of Mathematics and Statistics, Central South University, Changsha, Hunan, China.}
\email{youjundeng@csu.edu.cn; dengyijun\_001@163.com}

\author{Hongyu Liu}
\address{Department of Mathematics, Hong Kong Baptist University, Kowloon, Hong Kong SAR, China.}
\email{hongyu.liuip@gmail.com}

\author{Wing-Yan Tsui}
\address{Department of Mathematics, Hong Kong Baptist University, Kowloon, Hong Kong SAR, China.}
\email{wytsui.yan@gmail.com}

\date{} 
\begin{document}
\maketitle

\begin{abstract}
We are concerned with the inverse problem of identifying magnetic anomalies with varing parameters beneath the Earth using geomagnetic monitoring. Observations of the change in Earth's magnetic field--the secular variation--provide information about the anomalies as well as their variations. In this paper, we rigorously establish the unique recovery results for this magnetic anomaly detection problem. We show that one can uniquely recover the locations, the variation parameters including the growth or decaying rates as well as their material parameters of the anomalies. This paper extends the existing results in \cite{deng2018identifying} by two of the authors to varying anomalies.

\medskip

\noindent{\bf Keywords:}~Geomagnetic monitoring, magnetic anomaly detection, varying anomalies, unique recovery

\noindent{\bf 2010 Mathematics Subject Classification:}~~35Q60, 35J05, 31B10, 35R30, 78A40

\end{abstract}

\section{Introduction}

\subsection{Mathematical setup}

Focusing on the mathematics, but not the physics, we first present the mathematical setup of the inverse problem of identifying magnetic anomalies with varying parameters using geomagnetic monitoring.
We define the Earth as a core-shell structure with $\Sigma_c$ and $\Sigma$ signifying the core and the Earth, respectively.
It is assumed that both $\Sigma_c$ and $\Sigma$ are bounded on simply connected $C^2$ domains in $\mathbb{R}^3$ and $\Sigma_c\Subset \Sigma$. Then $\Sigma_s:=\Sigma\setminus\overline{\Sigma_c}$ signifies the shell of the Earth.
We do not restrict the Earth and its core as concentric balls in our study.
Throughout we let $\mathbf{x} = (x_j )^3_{j=1} \in \mathbb{R}^3$ denote the space variable. We characterize the electromagnetic (EM) medium by the electric permittivity $\epsilon(\mathbf{x})$, magnetic permeability $\mu(\mathbf{x})$ and electric conductivity $\sigma(\mathbf{x})$ in  $\mathbb{R}^3$. Let $\epsilon(\mathbf{x}),\ \mu(\mathbf{x})$ and $\sigma(\mathbf{x})$ be all real-valued scalar $L^\infty$ functions, such that $\epsilon(\mathbf{x})$ and $\mu(\mathbf{x})$ are positive, and $\sigma(\mathbf{x})$ is nonnegative. Let $\epsilon(\mathbf{x}) = \epsilon_0,\ \mu(\mathbf{x}) = \mu_0$ and $\sigma(\mathbf{x}) = 0$ for $\mathbf{x}\in \mathbb{R}^3\setminus\overline\Sigma$, where the two positive constants $\epsilon_0$ and $\mu_0$ signify the EM parameters of the uniformly homogeneous free space $\mathbb{R}^3\setminus\overline\Sigma$. Summarizing the above discussion, the medium distribution
is described by
\begin{equation}\label{materialdistfree}
\begin{split}
\sigma(\mathbf{x})&=\sigma_c(\mathbf{x})\chi(\Sigma_c),\\
\mu(\mathbf{x})&=(\mu_c(\mathbf{x})-\mu_0)\chi(\Sigma_c)+\mu_0,\\
\epsilon(\mathbf{x})&=(\epsilon_c(\mathbf{x})-\epsilon_0)+(\epsilon_s(\mathbf{x})-\epsilon_0)\chi(\Sigma_s)+\epsilon_0.
\end{split}
\end{equation}
where and also in what follows, $\chi$ denotes the characteristic function. \eqref{materialdistfree} characterizes the material configuration of the Earth's initial status.

Suppose there is a collection of magnetic anomalies which are growing or decaying in the shell of the Earth. Let the anomalies be represented as
$D_l\subset\mathbb{R}^3, l=1,2,\dots, l_0,$
where $D_l,1\leq l\leq l_0$ is a simply-connected Lipschitz domain such that the corresponding material parameters are described by $\epsilon_l,\ \mu_l$ and $\sigma_l$. It is assumed that $\epsilon_l,\ \mu_l$ and $\sigma_l$ are all positive constants with $\mu_l\neq\mu_0, 1 \leq l \leq l_0$. From a practical point of view, the size of the magnetic anomaly $(D_l;\epsilon_l,\mu_l,\sigma_l),  1\leq l\leq l_0$ is much smaller than the size of the Earth.  Hence, we can assume that the magnetic anomalies, with varying sizes, have the following forms
\begin{equation}\label{magnetized anomalies}
D_l=s_l\delta \Omega+\mathbf{z}_l,\quad  l=1,2,\dots, l_0,
\end{equation}
where $\Omega$ is a bounded Lipschitz domain in $\mathbb{R}^3$ with $\Omega\Subset\Sigma$, and $\delta\in \mathbb{R}_+$ with $\delta\ll 1$. $\mathbf{z}_l \in  \mathbb{R}^3, l=1,2,\dots, l_0$ signify the positions of the magnetic anomalies. Moreover, $s_l\in \mathbb{R}_+,  l=1,2,\dots, l_0$ signify the variation parameters of the magnetic anomalies such that
\begin{equation}\label{scaling factor}
	s_l=\delta^{\alpha_l},\quad  l=1,2,\dots, l_0,
\end{equation}
with $\alpha_l>-1,l=1,2,\dots, l_0$. Note that $\alpha_l>0$ means that the anomaly $D_l$ is decaying with a magnitude of $1+\alpha$, and $\alpha<0$ means that the anomaly $D_l$ is growing with a magnitude of $1+\alpha$, $l=1, 2, \ldots, l_0$. Throughout the rest of the paper, we assume that the distribution of $D_l, l=1,2,\dots, l_0$ is sparse, and $\mathbf{z}_l, l=1,2,\dots, l_0$ are far away from $\partial\Sigma$ such that $\mathbf{x}-\mathbf{z}_l\gg s_l\delta,l=1,2,\dots,l_0,$ for any  $\mathbf{x}\in\partial\Sigma$.
With the presence of the magnetic anomalies $(D_l ; \epsilon_l, \mu_l,\sigma_l),$  $l=1,2,\dots, l_0$ in the shell of the Earth, the EM medium configuration in $\mathbb{R}^3$ is then described by
\begin{equation}\label{eq:MaterialDisturbution}
\begin{split}
\sigma(\mathbf{x})=&\sigma_c(\mathbf{x})\chi(\Sigma_c)+\sum_{l=1}^{l_0}\sigma_l\chi(D_l),
\\ \mu(\mathbf{x})=&\left(\mu_c(\mathbf{x})-\mu_0\right)\chi(\Sigma_c)+\sum_{l=1}^{l_0}(\mu_l-\mu_0)\chi(D_l)+\mu_0,
\\
\epsilon(\mathbf{x})=&\left(\epsilon_c(\mathbf{x})-\epsilon_0\right)\chi(\Sigma_c)+\left(\epsilon_s(\mathbf{x})-\epsilon_0\right)\chi\left(\Sigma_s\setminus\overline{\bigcup_{l=1}^{l_0}D_l}\right)\\&+\sum_{l=1}^{l_0}(\epsilon_l-\epsilon_0)\chi(D_l)+\epsilon_0.
\end{split}
\end{equation}

Let $\mathbf{E}(\mathbf{x})$ and $\mathbf{H}(\mathbf{x})$ respectively signify the Earth's electric and magnetic fields, which satisfy the following Maxwell system in the time-harmonic regime for $(\mathbf{x},\omega)\in\mathbb{R}^3\times\mathbb{R}_+$  (cf. \cite{BPCo96,deng2018identifying}),
\begin{equation}\label{frequencymaxwell}
\begin{cases}
\nabla\times\mathbf{H}=-i\omega\epsilon\mathbf{E}+\sigma(\mathbf{E}+\mu\mathbf{v}\times\mathbf{H})&\text{in }\mathbb{R}^3,\smallskip\\
\nabla\times\mathbf{E}=i\omega\mu\mathbf{H}&\text{in }\mathbb{R}^3,\smallskip\\
\nabla\cdot(\mu\mathbf{H})=0,\quad\nabla\cdot(\epsilon\mathbf{E})=\hat{\rho}&\text{in }\mathbb{R}^3,\smallskip\\
\lim\limits_{\|x\|\rightarrow\infty}\|\mathbf{x}\|(\sqrt{\mu_0}\mathbf{H}\times\hat{\mathbf{x}}-\sqrt{\epsilon_0}\mathbf{E})=0.&
\end{cases}
\end{equation}
In \eqref{frequencymaxwell}, $\omega\in\mathbb{R}_+$ is the frequency, $\hat\rho(\mathbf{x})\in L^2(\Sigma_c)$ stands for the charge density of the Earth core $\Sigma_c$, and $\mathbf{v}\in L^\infty(\Sigma)$ is the fluid velocity of the Earth $\Sigma$. Here, $\hat\rho$ and $\mathbf{v}$ have the zero extension to $\mathbb{R}^3$. $\mathbf{v}\times \mu \mathcal{H}$ is the so-called motional electromotive force generated by the rotation of the Earth $\Sigma$. The last equation in \eqref{frequencymaxwell} is called the Silver-M\"uller radiation condition and it holds uniformly in all directions $\hat{\mathbf{x}}:=\frac{\mathbf{x}}{\|\mathbf{x}\|}\in\mathbb{S}^2$ for $\mathbf{x}\in \mathbb{R}^3\setminus\{0\}$. The Silver-M\"uller radiation condition characterizes the outgoing electromagnetic waves (cf.\cite{leis1986initial}), and also guarantees the well-posedness of a physical solution to \eqref{frequencymaxwell}. There exists a unique pair of solutions $(\mathbf{E},\mathbf{H}) \in H_{loc}(\mathrm{curl},\mathbb{R}^3)^2$ to \eqref{frequencymaxwell} (cf. \cite{LRX}). Here and also in what follows,  for any bounded domain $\mathcal{M}\subset \mathbb{R}^3$ with a Lipschitz boundary $\partial \mathcal{M}$,
\begin{equation}
H(\mathrm{curl}; \mathcal{M}):=\{U\in(L^2(\mathcal{M}))^3;\nabla\times U\in(L^2(\mathcal{M}))^3\}
\end{equation}
and
\begin{equation}
H_{loc}(\mathrm{curl};X):=\{U\vert_{\mathcal{M}}\in H(\mathrm{curl};\mathcal{M});\mathcal{M} \text{ is any bounded subdomain of } X\}.
\end{equation}

 Define $k_0 :=\omega\sqrt{\epsilon_0\mu_0}$ to be the wavenumber with respect to a frequency $\omega \in\mathbb{R}_+$. Then the far-field pattern is the asymptotic expansion of the corresponding scattered electric field or magnetic field as $\|\mathbf{x}\|\rightarrow+\infty$ (cf. \cite{colton1998inverse,nedelec2001acoustic})
\begin{equation}\label{asymptotic expansion electric}
\mathbf{E}(\mathbf{x})=\frac{e^{ik_0\|\mathbf{x}\|}}{\|\mathbf{x}\|}\mathbf{E}^\infty(\hat{\mathbf{x}})+O\left(\frac{1}{\|\mathbf{x}\|^2}\right),
\end{equation}
and
\begin{equation}\label{asymptotic expansion magnetic}
\mathbf{H}(\mathbf{x})=\frac{e^{ik_0\|\mathbf{x}\|}}{\|\mathbf{x}\|}\mathbf{H}^\infty(\hat{\mathbf{x}})+O(\frac{1}{\|\mathbf{x}\|^2}),
\end{equation}
which hold uniformly for all directions $\hat{\mathbf{x}}$.

Let $(\mathbf{E},\mathbf{H})$ be the solution to the Maxwell system \eqref{frequencymaxwell} associated to the medium configuration in \eqref{eq:MaterialDisturbution} with the variation parameters $\alpha_l=0,l=1,2,\dots,l_0$, and $(\mathbf{E}_s,\mathbf{H}_s)$ be the solution to \eqref{frequencymaxwell} associated to \eqref{eq:MaterialDisturbution} with the variation parameters $\alpha_l\neq0,l=1,2,\dots,l_0$. That is, $(\mathbf{E},\mathbf{H})$ is the geomagnetic field of the Earth before the variation of the magnetic anomalies, whereas $(\mathbf{E}_s,\mathbf{H}_s)$ is the geomagnetic field after the change of the anomalies. Let $\Gamma$ be an open surface located away from $\Sigma$, and it signifies the location of the receivers for the geomagnetic monitoring.

 With the above preparations, the inverse problem for the magnetic anomaly detection can be reformulated as
\begin{equation}\label{inverse problem frequency}
(\mathbf{H}_s^\infty(\hat{\mathbf{x}};\omega)-\mathbf{H}^\infty(\hat{x};\omega))\Big\vert_{(\hat{\mathbf{x}},\omega)\in\widehat{\Gamma}\times\mathbb{R}_+}\longrightarrow\bigcup_{l=1}^{l_0}(D_l; \mu_l,\sigma_l,\alpha_l,\mathbf{z}_l),
\end{equation}
where $\widehat{\Gamma}:=\big\{\hat{\mathbf{x}}\in\mathbb{S}^2;\hat{\mathbf{x}}=\frac{\mathbf{x}}{\|\mathbf{x}\|},\mathbf{x}\in \Gamma\}$.
That is, by monitoring the change of Earth's geomagnetic field, one intends to recover the locations, material parameters as well as the size variations of the anomalies. By Rellich's theorem, there is a one-to-one correspondence between the pair of the far-field patterns $(\mathbf{E}^\infty,\mathbf{H}^\infty)$ and the EM wave $(\mathbf{E},\mathbf{H})$ (cf. \cite{colton1998inverse} ), and similar for the electromagnetic fields $(\mathbf{E}_s^\infty,\mathbf{H}_s^\infty)$ and $(\mathbf{E}_s,\mathbf{H}_s)$, respectively. By the real analyticity of $\mathbf{H}^\infty$ and $\mathbf{H}_s^\infty$, the inverse problem \eqref{inverse problem frequency} is equivalent to the following one,
\begin{equation}\label{inverseproblem FinalFequency}
(\mathbf{H}_s(\mathbf{x},\omega)-\mathbf{H}(\mathbf{x},\omega))\Big\vert_{(\mathbf{x},\omega)\in\partial \Sigma\times\mathbb{R}_+}\rightarrow\bigcup_{l=1}^{l_0}(D_l; \mu_l,\sigma_l,\alpha_l,\mathbf{z}_l).
\end{equation}

\subsection{Physical background and connection to existing results}

Geomagnetic detections have been successfully applied in science and engineering including military submarine detection, minerals searching and geoexploration (cf. \cite{W1}). In fact, the operations of the Swarm satellites, a set of three probes sent by the European Space Agency (ESA) in the year 2013, are capable of measuring variations in the Earth's magnetic field. These so-called secular variations can be detected by studying what happens beneath the Earth.

In geophysics, one widely accepted but relatively simpler model of Earth's geomagnetic field is described by a time-dependent Maxwell model as follows.
Let $\mathcal{E}(\mathbf{x}, t)$ and $\mathcal{H}(\mathbf{x}, t), (\mathbf{x}, t) \in \mathbb{R}^3 \times \mathbb{R}_+$, denote the electric and magnetic fields of the Earth $\Sigma$, and they satisfy the following time-dependent Maxwell system for $(\mathbf{x},t)\in\mathbb{R}^3\times\mathbb{R}_+$:

\begin{equation}\label{eq:maxwell system}
\begin{cases}
\nabla\times\mathcal{H}({\bf x},t)=\epsilon({\bf x})\partial_t\mathcal{E}({\bf x},t)+\sigma({\bf x})\left(\mathcal{E}({\bf x},t)+\mu({\bf x})\mathbf{v}\times \mathcal{H}({\bf x},t)\right),\smallskip\\
\nabla\times\mathcal{E}(\mathbf{x},t)=-\mu(\mathbf{x})\partial_t\mathcal{H}(\mathbf{x},t),\smallskip\\
\nabla\cdot(\mu(\mathbf{x})\mathcal{H}(\mathbf{x},t))=0,\smallskip\\
\nabla\cdot(\epsilon(\mathbf{x})\mathcal{E}(\mathbf{x},t))=\rho(\mathbf{x},t),\smallskip\\
\mathcal{E}(\mathbf{x},0)=\mathcal{H}(\mathbf{x},0)=0,
\end{cases}
\end{equation}
where $\rho(\mathbf{x},t)=\rho_c (\mathbf{x},t)\chi(\Sigma_c)\in H^1(L^2(\Sigma_c),\mathbb{R}_+)$ stands for the charge density of the Earth core $\Sigma_c$, and $\mathbf{v}(\mathbf{x})\in L^\infty(\Sigma)$ is the fluid velocity of the Earth $\Sigma$.

Let $(\mathcal{E},\mathcal{H})$ be the geomagnetic field of the Earth before the variation of the magnetic anomalies, whereas $(\mathcal{E}_s,\mathcal{H}_s)$ be the geomagnetic field after the change of the anomalies. That is, $(\mathcal{E},\mathcal{H})$ is the solution to \eqref{eq:maxwell system} associated to the medium configuration in \eqref{eq:MaterialDisturbution} with the variation parameters $\alpha_l=0,l=1,2,\dots,l_0$, and $(\mathbf{E}_s,\mathbf{H}_s)$ is the solution to \eqref{eq:maxwell system} associated to \eqref{eq:MaterialDisturbution} with the variation parameters $\alpha_l\neq0,l=1,2,\dots,l_0$
In such a setup, the magnetic anomaly detection can be formulated as follows,
\begin{equation}\label{inverseproblemTime}
(\mathcal{H}_s(\mathbf{x},t)-\mathcal{H}(\mathbf{x},t))\Big\vert_{(\mathbf{x},t)\in\Gamma\times\mathbb{R}_+}\rightarrow\bigcup_{l=1}^{l_0}( D_l; \mu_l,\sigma_l,\alpha_l,\mathbf{z}_l).
\end{equation}
By introducing the following temporal Fourier transform for $(\mathbf{x},\omega)\in\mathbb{R}^3\times\mathbb{R}_+$:
\begin{equation}\label{temporal Fourier transform}
\mathbf{J}(\mathbf{x},\omega)=\mathcal{F}_t(\mathcal{J}):=\frac{1}{2\pi}\int_{0}^{\infty}\mathcal{J}(\mathbf{x},t)e^{i\omega t}dt,\quad\mathcal{J}=\mathcal{E}\text{ or }\mathcal{H},
\end{equation}
and setting
\begin{equation*}
\begin{split}
\mathbf{E}=\mathcal{F}_t(\mathcal{E}),\quad\mathbf{H}=\mathcal{F}_t(\mathcal{H}),&\quad\partial_t\mathbf{E}=-i\omega\mathcal{F}_t(\mathcal{E}),\quad \partial_t\mathbf{H}=-i\omega\mathcal{F}_t(\mathcal{H}),\\ \mathbf{E}_s=\mathcal{F}_t(\mathcal{E}_s),\quad\mathbf{H}_s=\mathcal{F}_t(\mathcal{H}_s),\quad\partial_t\mathbf{E}_s&=-i\omega\mathcal{F}_t(\mathcal{E}_s),\quad \partial_t\mathbf{H}_s=-i\omega\mathcal{F}_t(\mathcal{H}_s),\quad\hat{\rho}=\mathcal{F}_t(\rho),
\end{split}
\end{equation*}
one can show by direct verifications that \eqref{eq:maxwell system} is reduced to \eqref{frequencymaxwell}, and accordingly, the inverse problem \eqref{inverseproblemTime} is converted into \eqref{inverse problem frequency}; we also refer to \cite{deng2018identifying} for more relevant discussion about this reduction.

In what follows, we would like to make a few remarks regarding the mathematical formulation of the inverse problem \eqref{inverseproblemTime}, or equivalently \eqref{inverseproblem FinalFequency}. First, we emphasize that in \eqref{inverseproblem FinalFequency} or \eqref{inverseproblemTime}, we do not assume the medium configuration of the Earth's core, the charge density and the fluid velocity to be a priori known. This make our study more practical, but more challenging on the other hand. Second, the measurement surface $\Gamma$ in \eqref{inverseproblemTime} is in a scale much smaller than the Earth, and we are mainly concerned with the region where the corresponding geomagnetic variation can reach $\Gamma$. Hence, the rest part of the Earth's medium configuration can be assumed to be the same to that of the aforementioned region. Therefore, it is unobjectionable to claim that the medium configuration in \eqref{eq:MaterialDisturbution} is a realistic one. Third, the variation of the magnetic anomalies is assumed to take place in a very long time scale, i.e., the variation process occurs very slowly. In fact, one of the practical scenarios that motivates the current study is the discovery of a giant underground iron river in 2016 by observations of the change of Earth's geomagnetic field; see \cite{LHF,N1,N2}. Hence, one can have the Fourier transforms of the geomagnetic fields $\mathcal{H}_s$ and $\mathcal{H}$ separately. By a similar justification, the infinite time interval for the measurement in \eqref{inverseproblemTime} can be replaced by a finite one $[0,T_0]$ for some finite $T_0\in\mathbb{R}_+$, since the outgoing geomagnetic variation eventually leaves the Earth after a finite time.


The magnetic anomaly detection was mathematically investigated \cite{cdengMHD,deng2018identifying} by two of the authors of the present article. In \cite{deng2018identifying}, a Maxwell system is used for the the geomagnetic model, whereas in \cite{cdengMHD}, a magnetohydrodynamic system is used for the geomagnetic model. Both \cite{deng2018identifying} and \cite{cdengMHD} deal with the case that the geomagnetic variation is caused by the presence of certain magnetized anomalies, which e.g. corresponds to a specific practical scenario of submarine detection. This is in a sharp difference to the present study, where the geomagnetic variation is caused by the change of the magnetic anomalies. That is, the magnetic anomaly already exists, and its growth or shrinkage causes the geomagnetic variation. In order to establish the unique recovery results in such a new situation, we basically follow the general strategy developed in \cite{cdengMHD,deng2018identifying} by combining low-frequency asymptotics and various linearization techniques. However, there are several new challenges that require some new technical developments. It is emphasized that in \eqref{inverseproblem FinalFequency}, we do not assume that the Earth's geomagnetic field in the case with no anomaly is known a priori. Hence, one cannot naively reduce the current study to the one considered in \cite{cdengMHD,deng2018identifying}, where the knowledge of the free geomagnetic field is a critical ingredient. Finally, we would like to point out that our uniqueness argument is constructive, and it actually implies a construction procedure for the anomaly detection.


The rest of the paper is organized as follows. In Section 2, we derive the low-frequency asymptotic expansions of the geomagnetic fields as well as further asymptotic expansions of the static fields with respect to the variation parameters of the anomalies. We basically refine the asymptotic arguments in \cite{deng2018identifying} by further analyzing the linearized fields. Section 3 is devoted to the recovery of the free geomagnetic field. In Section 4, we establish the main unique recovery results on identifying the positions, the variation parameters as well as the material parameters of the magnetic anomalies.

\section{Asymptotic expansions of the geomagnetic fields }

For our subsequent study of the inverse problem \eqref{inverse problem frequency}, or equivalently \eqref{inverseproblem FinalFequency}, we derive in this section several critical asymptotic expansions of the geomagnetic fields $\mathbf{H}$ and $\mathbf{H}_s$, respectively, associated with \eqref{eq:MaterialDisturbution}--\eqref{frequencymaxwell}. These include the low-frequency asymptotic expansions of the geomagnetic fields in obtaining the corresponding static fields, and further asymptotic expansions of the static fields with respect to the anomaly sizes. We make essential use of tools from the layer potential theory, which we shall first collect in the following.

\subsection{Layer potentials}
 We recall some preliminary knowledge on the layer potential techniques for the subsequent use. Let $B\subset \mathbb{R}^3$ be a bounded Lipschitz domain. We introduce some function spaces on the boundary $\partial B$. Let us first define the space of tangential vector fields by
\begin{equation}
L^2_T(\partial B):=\{\Phi\in L^2(\partial B)^3,\nu\cdot\Phi=0\}.
\end{equation}
We denote by $\nabla_{\partial B}\cdot$ the standard surface divergence operator defined on $L^2_T(\partial B)$. Let $H^s(\partial B)$ be the usual Sobolev space of order $s\in \mathbb{R}$ on $\partial B$. Define the normed spaces of tangential fields by
\begin{equation}
\begin{aligned}
\mathrm{TH}(\mathrm{div},\partial B)&:=\left\{\Phi\in L^2_T(\partial B):\nabla_{\partial B}\cdot \Phi \in L^2(\partial B)\right\} ,\\
\mathrm{TH}(\mathrm{curl},\partial B)&:=\left\{\Phi\in L^2_T(\partial B):\nabla_{\partial B}\cdot (\Phi\times\nu) \in L^2(\partial B)\right\},
\end{aligned}
\end{equation}
endowed with the norms
\begin{equation}
\begin{aligned}
\|\Phi\|_{\text{TH(div,}\partial B)}&=\|\Phi\|_{L^2(\partial B)}+\|\nabla_{\partial B}\cdot\Phi\|_{L^2(\partial B)},\\
\|\Phi\|_{\text{TH(div,}\partial B)}&=\|\Phi\|_{L^2(\partial B)}+\|\nabla_{\partial B}\cdot\Phi\|_{L^2(\partial B)},
\end{aligned}
\end{equation}
respectively.

Let $k\in\mathbb{R}_+$ to be a wavenumber. We recall that the outgoing fundamental solution $\Gamma_k$ to the PDO $(\Delta +k^2)$ in $\mathbb{R}^3$ is given by
\begin{equation}\label{fundamental solution}
\Gamma_k(\mathbf{x})=-\frac{e^{ik\|\mathbf{x}\|}}{4\pi\|\mathbf{x}\|},\quad\mathbf{x}\in\mathbb{R}^3\ \text{and}\ \mathbf{x}\neq\vect{0}.
\end{equation}
For a scalar density $\phi\in L^2(\partial B)$, let $S_B^k: H^{-1/2}(\partial B)\rightarrow H^{1}(\mathbb{R}^3\setminus\partial B)$ be the single layer potential associated with $B\subset\mathbb{R}^3$,
\begin{equation}\label{singlelayerpotential}
S^k_B[ \phi](\mathbf{x}):=\int_{\partial B}\Gamma_k(\mathbf{x}-\mathbf{y})\phi(\mathbf{y})ds_{\mathbf{y}},\quad \mathbf{x}\in \mathbb{R}^3\setminus\partial B,
\end{equation}
and $\mathcal{K}_B^k:H^{1/2}(\partial B)\rightarrow H^{1/2}(\partial B)$ be the Neumann-Poincar\'e operator
\begin{equation}\label{Neumann-Poincace operator }
\mathcal{K}_B^k[\phi](\mathbf{x}):=\mathrm{p.v.}\int_{\partial B}\frac{\partial\Gamma_k(\mathbf{x}-\mathbf{y})}{\partial \nu_{\mathbf{y}}}\phi(\mathbf{y})ds_{\mathbf{y}},\quad \mathbf{x}\in\partial B,
\end{equation}
where p.v. stands for the Cauchy principle value. In \eqref{Neumann-Poincace operator } and also in what follows, unless otherwise specified, $\nu$  signifies the exterior unit normal vector to the boundary of the concerned domain. The single layer potential $S^k_B $ satisfies the following trace formula
\begin{equation}\label{trace formula}
\frac{\partial}{\partial\nu}S^k_B[\phi]\Big\vert_\pm=\left(\pm \frac{1}{2}I+(\mathcal{K}^k_B)^*\right)[\phi]\quad\text{on}\ \partial B,
\end{equation}
where $(\mathcal{K}^k_B)^*$ is the adjoint operator of $\mathcal{K}^k_B$ with respect to $L^2$ inner product.

In addition, for a density $\Phi \in  \mathrm{TH}
(\mathrm{div}, \partial B)$, we define the vectorial single layer potential by
\begin{equation}
\mathcal{A}_B^k[\Phi](\mathbf{x}):=\int_{\partial B}\Gamma_k(\mathbf{x}-\mathbf{y})\Phi(\mathbf{y})ds_{\mathbf{y}},\quad\mathbf{x}\in\mathbb{R}^3\setminus\partial B.
\end{equation}
It is known that $\nabla\times \mathcal{A}^k_B$ satisfies the jump formula (see \cite{Ammari2016surface,nedelec2001acoustic})
\begin{equation}
\nu\times\nabla\times\mathcal{A}_B^k[\Phi]\Big\vert_\pm=\mp\frac{\Phi}{2}+\mathcal{M}^k_B[\Phi]\quad \text{on }\partial B,
\end{equation}
where
\begin{equation}
\forall\mathbf{x}\in \partial B, \quad\nu\times\nabla\times\mathcal{A}^k_B[\Phi]\Big\vert_\pm(\mathbf{x})=\lim_{t\rightarrow 0^+}\nu\times\nabla\times\mathcal{A}^k_B[\Phi](\mathbf{x}\pm t\nu),
\end{equation}
is understood in the sense of uniform convergence on $\partial B$ and  the boundary integral operator $\mathcal{M}^k_B: \text{TH(div,} \partial B) \rightarrow \text{TH(div,} \partial B) $ is given by
\begin{equation}
\mathcal{M}^k_B[\Phi](\mathbf{x})=\text{p.v.}\quad\nu\times\nabla\times\int_{\partial B}\Gamma_k(\mathbf{x}-\mathbf{y})\Phi(\mathbf{y})ds_{\mathbf{y}}.
\end{equation}
We also define $\mathcal{L}^k_B$ : TH (div,$\partial B)\rightarrow$ TH (div,$\partial B$) by
\begin{equation}
\mathcal{L}^k_B[\Phi](\mathbf{x}):=\nu_{\mathbf{x}}\times\nabla\times\nabla\mathcal{A}^k_B[\Phi](\mathbf{x})=\nu_{\mathbf{x}}\times(k^2\mathcal{A}^k_B[\Phi](\Phi)+\nabla\nabla\cdot\mathcal{A}^k_B[\Phi]),
\end{equation}
where the formula $\nabla\times\nabla\times=-\Delta+\nabla\nabla\cdot$ is used and shall also be frequently used in what follows.

It is mentioned that $\frac{I}{2}\pm \mathcal{M}^k_B$ is invertible on TH(div,$\partial B$) when $k$ is sufficiently small (see e.g. \cite{Ammari2016surface,torres1998maxwell}). In the following , if $k=0$, we formally set $k$ introduced in \eqref{fundamental solution} to be $\Gamma_0=-\frac{1}{4\pi \|\mathbf{x}\|}$, and the other integral operators introduced above can also be formally defined when $k=0$. Finally, we define $k_s:=\omega\sqrt{\mu_0\epsilon_s}$.

\subsection{Low-frequency asymptotic expansions of the magnetic fields}

Through out the rest of the paper, we let $(\mathbf{E}_0,\mathbf{H}_0)$ be the solution to \eqref{materialdistfree} and \eqref{frequencymaxwell}. That is, $\mathbf{E}_0$ and $\mathbf{H}_0$ are the free/background geomagnetic field of the Earth. In this subsection, we derive the low-frequency asymptotic expansions of the magnetic fields $\mathbf{H}_0$ and $\mathbf{H}_s$. In the following, we define the wavenumbers $\zeta_l, l=1,2,\dots, l_0$ by $\zeta_l^2:=\omega^2\mu_l\gamma_l,\gamma_l:=\epsilon_l+i\sigma_l/\omega$, where the sign of $\zeta_l$ is chosen such that  $\Im \zeta_l\geq0$. For the sake of simplicity, we set $\tilde{D} := \Sigma_s \setminus\overline{\bigcup_{l{'}=1}^{l_0}{ D_{l{'}}}}$. We also let $\nu_l$ be the exterior unit normal vector defined on $\partial D_l ,  l=1,2,\dots, l_0$.

We next deal with the low-frequency asymptotic analysis of the geomagnetic system \eqref{frequencymaxwell}. As remarked earlier, we shall basically follow the general strategy developed in \cite{deng2018identifying}, but there are also several new technical developments in the current setup of out study. We would also like to mention in passing some related literature on the low-frequency asymptotic analysis of the Maxwell system \cite{A2,A3,Das02,deng2018identifying,cdengMHD,DLL19,DLU19,Kle}.

We first deal with $\mathbf{H}_0$, which can be represented by the following integral ansatz,
\begin{equation}
	\mathbf{H}_0=\begin{cases}
	-i\omega^{-1}(\nabla\times\nabla\times\mathcal{A}^{k_0}_\Sigma[\Phi_c]+\nabla\times\nabla\times\mathcal{A}^{k_0}_\Sigma[\Phi_s]+\omega^2\epsilon_0\nabla\times\mathcal{A}^{k_0}_\Sigma[\Psi_s])&\text{in }\mathbb{R}^3\setminus\overline\Sigma,\\
	-i\omega^{-1}(\nabla\times\nabla\times\mathcal{A}^{k_s}_\Sigma[\Phi_c]+\nabla\times\nabla\times\mathcal{A}^{k_s}_\Sigma[\Phi_s]+\omega^2\epsilon_0\nabla\times\mathcal{A}^{k_s}_\Sigma[\Psi_s])&\text{in }\Sigma_s,
	\end{cases}
\end{equation}
where $(\Phi_c,\Phi_s,\Psi_s)\in\text{TH(div,}\partial\Sigma_c)\times\text{TH(div,}\partial\Sigma_c)\times\text{TH(div,}\partial\Sigma_c)$. There holds the following asymptotic expansion,

\begin{lem}{(Lemma 2.3 in \cite{deng2018identifying})}
	Let $\mathbf{H}_0$ be the solution to \eqref{materialdistfree} and \eqref{frequencymaxwell}. Then for $\omega\in \mathbb{R}_+$ sufficiently small, one has
	\begin{equation}
		\mathbf{H}_0=\nabla S^0_{\Sigma_c}(\frac{I}{2}+(\mathcal{K}^0_{\Sigma_c})^*)^{-1}[\nu\cdot\mathbf{H}_0\vert^+_{\partial\Sigma_c}]+\mathcal{O}(\omega)\quad\text{in }\mathbb{R}^3\setminus\overline\Sigma_c.
	\end{equation}
\end{lem}
Next we establish the asymptotic expansion of the magnetic field $\mathbf{H}$ associated with the system \eqref{eq:MaterialDisturbution} and \eqref{frequencymaxwell}. By using the transmission conditions across $\partial \Sigma$ and $\partial D_l, l=1,2,\dots, l_0$ , $(\mathbf{E}_s,\mathbf{H}_s)$ is the solution to the following transmission problem:
\begin{equation}
\begin{cases}
\nabla\times\mathbf{H}_s=-i\omega\epsilon_c\mathbf{E}_s+\sigma_c(\mathbf{E}_s+\mu_c\mathbf{v}\times\mathbf{H}_s),&\text{in }\Sigma_c,\\
\nabla\times\mathbf{E}_s=i\omega\mu_c\mathbf{H}_s,&\text{in }\Sigma_c,\\
\nabla\cdot(\mu_c\mathbf{H}_s)=0,\quad \nabla\cdot(\epsilon_c\mathbf{E}_s)=\hat{\rho},&\text{in }\Sigma_c,\\
[\nu\times\mathbf{E}_s]=[\nu\times\mathbf{H}_s]=0,&\text{on } \partial \Sigma_c,\\
\nabla\times\mathbf{H}_s=-i\omega\epsilon_s\mathbf{E}_s,\quad\nabla\times\mathbf{E}^s=i\omega\mu_0\mathbf{H}_s,&\text{in }\tilde{D},\\
[\nu\times\mathbf{E}_s]=[\nu\times\mathbf{H}_s]=0,&\text{on } \partial \Sigma,\\
\nabla\times\mathbf{H}_s=-i\omega\displaystyle\sum_{l=1}^{l_0}\epsilon_l\mathbf{E}_s+\displaystyle\sum_{l=1}^{l_0}\sigma_l\mathbf{E}_s,&\text{in }D_l, l=1,2,\dots, l_0,\\
\nabla\times\mathbf{E}_s=i\omega\displaystyle\sum_{l=1}^{l_0}\mu_l\mathbf{H}_s,&\text{in }D_l, l=1,2,\dots, l_0,\\
[\nu\times\mathbf{E}_s]=[\nu\times\mathbf{H}_s]=0,&\text{on } \partial {D_l}, l=1,2,\dots, l_0,\\
\nabla\times\mathbf{H}_s=-i\omega\epsilon_0\mathbf{E}_s,\quad\nabla\times\mathbf{E}_s=i\omega\mu_0\mathbf{H}_s,&\text{in }\mathbb{R}^3\setminus\overline{\Sigma},\\
\lim\limits_{\|x\|\rightarrow\infty}\|\mathbf{x}\|(\sqrt{\mu_0}\mathbf{H}_s\times\hat{\mathbf{x}}-\sqrt{\epsilon_0}\mathbf{E}_s)=0,
\end{cases}
\end{equation}
where $[\nu\times\mathbf{E}_s]$ and $[\nu\times\mathbf{H}_s]$ denote the jumps of $\nu\times\mathbf{E}_s$ and $\nu\times\mathbf{H}_s$ along $\partial \Sigma$ and $\partial {D_l}$, namely,
\begin{equation}
[\nu\times\mathbf{E}_s]=(\nu\times\mathbf{E}_s)\vert_+-(\nu\times\mathbf{E}_s)\vert_-,\quad	[\nu\times\mathbf{H}_s]=(\nu\times\mathbf{H}_s)\vert_+-(\nu\times\mathbf{H}_s)\vert_-.
\end{equation}
Using the potential theory, the solution $(\mathbf{E}_s,\mathbf{H}_s)$ can be represented by the following formulas
\begin{equation}\label{field of E with magnetized anomalies}
\mathbf{E}_s=\begin{cases}
\hat{\mathbf{E}}_0+\nabla\times\left(\mu_0\mathcal{A}^{k_0}_\Sigma[\Phi_0]+\nabla\times\mathcal{A}^{k_0}_\Sigma[\Psi_0]\right)\\
\quad+\nabla\times\sum_{l{'}=1}^{l_0}\left(\mu_0\mathcal{A}^{k_0}_{D_{l{'}}}[\Phi_{l'}]+\nabla\times\mathcal{A}^{k_0}_{D_{l{'}}}[\Psi_{l'}]\right),&\text{in}\quad \mathbb{R}^3\setminus\overline{\Sigma},\\
\hat{\mathbf{E}}_0+\nabla\times\left(\mu_0\mathcal{A}^{k_s}_\Sigma[\Phi_0]+\nabla\times\mathcal{A}^{k_s}_\Sigma[\Psi_0]\right)\\
\quad+\nabla\times\sum_{l{'}=1}^{l_0}\left(\mu_0\mathcal{A}^{k_s}_{D_{l{'}}}[\Phi_{l{'}}]+\nabla\times\mathcal{A}^{k_s}_{D_{l{'}}}[\Psi_{l{'}}]\right),&\text{in}\quad \tilde{D},\\
\nabla\times\left(\mu_l\mathcal{A}^{\zeta_l}_\Sigma[\Phi_0]+\nabla\times\mathcal{A}^{\zeta_l}_\Sigma[\Psi_0] \right)\\
\quad+\nabla\times\sum_{l^{'}=1}^{l_0}\left(\mu_l\mathcal{A}^{\zeta_l}_\Sigma[\Phi_{l^{'}}]+\nabla\times\mathcal{A}^{\zeta_l}_\Sigma[\Psi_{l^{'}}]\right),&\text{in}\quad D_l,
\end{cases}
\end{equation}
and
\begin{equation}\label{field of H with magnetized anomalies}
\mathbf{H}_s=\begin{cases}
\hat{\mathbf{H}}_0-i\omega^{-1}\nabla\times\Big(\left(\omega^2\epsilon_0\mathcal{A}^{k_0}_\Sigma[\Psi_0]+\nabla\times\mathcal{A}^{k_0}_\Sigma[\Phi_0]\right)\\
\quad+\sum_{l{'}=1}^{l_0}\left(\omega^2\epsilon_0\mathcal{A}^{k_0}_{D_{l{'}}}[\Psi_{l'}]+\nabla\times\mathcal{A}^{k_0}_{D_{l{'}}}[\Phi_{l'}]\right)\Big),&\text{in}\quad \mathbb{R}^3\setminus\overline{\Sigma},\\
\hat{\mathbf{H}}_0-i\omega^{-1}\nabla\times\Big(\left(\omega^2\epsilon_s\mathcal{A}^{k_s}_\Sigma[\Psi_0]+\nabla\times\mathcal{A}^{k_s}_\Sigma[\Phi_0]\right)\\
\quad+\sum_{l{'}=1}^{l_0}\left(\omega^2\epsilon_s\mathcal{A}^{k_s}_{D_{l{'}}}[\Psi_{l'}]+\nabla\times\mathcal{A}^{k_s}_{D_{l{'}}}[\Phi_{l'}]\right)\Big),&\text{in}\quad \tilde{D},\\
-i\omega^{-1}\nabla\times\Big(\left(\omega^2\gamma_l\mathcal{A}^{\zeta_l}_\Sigma[\Psi_0]+\nabla\times\mathcal{A}^{\zeta_l}_\Sigma[\Phi_0]\right)\\
\quad+\sum_{l{'}=1}^{l_0}\left(\omega^2\gamma_{l'}\mathcal{A}^{\zeta_{l'}}_{D_{l{'}}}[\Psi_{l'}]+\nabla\times\mathcal{A}^{\zeta_{l'}}_{D_{l{'}}}[\Phi_{l'}]\right)\Big),&\text{in}\quad D_l,
\end{cases}
\end{equation}
where by transmission conditions across $\partial \Sigma$ and $\partial D_l, l=1,2,\dots, l_0$ , $(\Phi_0,\Psi_0)\in \text{TH(div,}\partial\Sigma)\times\text{TH(div},\partial\Sigma)$ and $(\Phi_l, \Psi_l)\in \text{TH(div,}\partial D_l)\times\text{TH(div},\partial D_l),l=1,2,\dots,l_0 $ satisfy
\begin{equation}\label{transmission conditions}
\begin{split}
&\left(-\mu_0I+\mathcal{D}^{k_0,k_s}_{\Sigma,\Sigma}\right)[\Phi_0]+\mathcal{P}^{k_0,k_s}_{\Sigma,\Sigma}[\Psi_0]\\
=&
\sum_{l{'}=1}^{l_0}\left(\mathcal{D}^{k_s,k_0}_{\Sigma,{D_{l'}}}[\Phi_{l'}]+\mathcal{P}^{k_s,k_0}_{\Sigma,{D_{l{'}}}}[\Psi_{l'}]\right), \quad \text{on}\quad\partial\Sigma,\\
&\omega^2\left(-\frac{\epsilon_s+\epsilon_0}{2}I+\mathcal{F}^{k_0,k_s}_{\Sigma,\Sigma}\right)[\Psi_0]+\mathcal{P}^{k_0,k_s}_{\Sigma,\Sigma}[\Phi_0]\\
=&\sum_{l^{'}=1}^{l_0}\left(\omega^2\mathcal{F}^{k_s,k_0}_{\Sigma,D_{l'}}[\Psi_{l{'}}]+\mathcal{P}^{k_s,k_0}_{\Sigma,D_{l'}}[\Phi_{l{'}}]\right),\quad\text{on}\quad\partial\Sigma,\\
&\left(\frac{\mu_l+\mu_0}{2}I+\mathcal{D}^{\zeta_l,k_s}_{D_l,D_l}\right)[\Phi_l]+\sum_{l^{'}\ne l}^{l_0}\mathcal{D}^{\zeta_l,k_s}_{D_l,D_{l'}}[\Phi_{l'}]+\sum_{l^{'}=1}^{l_0}\mathcal{P}^{\zeta_l,k_s}_{D_l,D_{l'}}[\Psi_{l'}]\\
=&\nu_l\times\hat{\mathbf{E}}_0+\mathcal{D}^{k_s,\zeta_l}_{D_l,\Sigma}[\Phi_{0}]+\mathcal{P}^{k_s,\zeta_l}_{D_l,\Sigma}[\Psi_{0}],\quad\text{on}\quad\partial D_l,\\
&\omega^2\left(\frac{\gamma_l+\epsilon_s}{2}I+\mathcal{F}^{\zeta_l,k_s}_{D_l,D_l}\right)[\Psi_l]+\omega^2\sum_{l^{'}\neq l}^{l_0}+\mathcal{F}^{\zeta_l,k_s}_{D_l,D_{l'}}[\Psi_{l'}]+\sum_{l^{'}=1}^{l_0}\mathcal{P}^{\zeta_l,k_s}_{D_l,D_{l'}}[\Phi_{l'}]\\
=&i\omega\nu_l\times\hat{\mathbf{H}}_0+\omega^2\mathcal{F}^{k_s,\zeta_l}_{D_l,\Sigma}[\Psi_{0}]+\mathcal{P}^{k_s,\zeta_l}_{D_l,\Sigma}[\Phi_{0}],\quad\text{on}\quad\partial D_l,
\end{split}
\end{equation}
with the operators $\mathcal{P}^{k,k'}_{D,D'}$, $\mathcal{D}^{k,k'}_{D,D'}$ and $\mathcal{F}^{k,k'}_{D,D'}$ defined by
\begin{equation}
	\begin{aligned}
	\mathcal{P}^{k,k'}_{D,D'}:=&\mathcal{L}^{k}_{D,D'}-\mathcal{L}^{k'}_{D,D'},\\
	\mathcal{D}^{k,k'}_{D,D'}:=&\mu^{(k)}\mathcal{M}^{k}_{D,D'}-\mu^{(k')}\mathcal{M}^{k'}_{D,D'},\\
	\mathcal{F}^{k,k'}_{D,D'}:=&\gamma^{(k)}\mathcal{M}^{k}_{D,D'}-\gamma^{(k')}\mathcal{M}^{k'}_{D,D'},
	\end{aligned}
\end{equation}
and the operators $\mathcal{M}^{k}_{D,D'}$ and $\mathcal{L}^{k}_{D,D'}$ defined by
\begin{equation}
	\begin{aligned}
	\mathcal{M}^{k}_{D,D'}:=&\nu\times\nabla\times \mathcal{A}^{k}_{D'}\vert_{\partial D},\quad \mathcal{M}^{k}_{D}:=\mathcal{M}^{k}_{D,D}\\
	\mathcal{L}^{k}_{D,D'}:=&\nu\times\nabla\times\nabla\times \mathcal{A}^{k}_{D'}\vert_{\partial D},
	\end{aligned}
\end{equation}
where $k,k'\in\{k_0,k_s,\zeta_1,\zeta_2,\dots,\zeta_{l_0}\},D,D{'}\in \{\Sigma, D_1,D_2,\dots,D_{l_0}\}$ and $\mu^{(k)},\gamma^{(k)}$ are respectively the parameters $\mu$ and $\gamma$, which are related to $k$. In other words,
\begin{equation}
\begin{split}
\mu^{(k_0)}=\mu^{(k_s)}=\mu_0,\quad\mu^{(\zeta_l)}=\mu_l,\\
\gamma^{(k_0)}=\epsilon_0,\quad\gamma^{(k_s)}=\epsilon_s,\quad\gamma^{(\zeta_l)}=\zeta_l.
\end{split}
\end{equation}
The fields $(\hat{\mathbf{E}}_0, \hat{\mathbf{H}}_0)$ in \eqref{field of E with magnetized anomalies} and \eqref{field of H with magnetized anomalies} satisfy \eqref{eq:MaterialDisturbution} and \eqref{frequencymaxwell} in $(\mathbb{R}^3\setminus\overline{\Sigma})\bigcup \tilde{D}$ with $\hat{\mathbf{H}}_0$ given in Lemma 2.1 in \cite{deng2018identifying} (see Lemma \ref{le:0101} below) and they depend on the background fields $({\mathbf{E}}_0, {\mathbf{H}}_0)$ and the boundary condition on $\partial \Sigma_c$.

Based on the above integral representation, we present the asymptotic formula of the magnetic field $\mathbf{H}_s$ with respect to the frequency. To that end, we first define an $l_0$-by-$l_0$ matrix type operator $\mathbb{K}^*_D$ on $L^2(\partial D_1)\times L^2(\partial D_2)\times\cdots\times L^2(\partial D_{l_0})$ by
\begin{equation}\label{K^*_D}
\mathbb{K}^*_D:=\begin{pmatrix}
(\mathcal{K}^0_{D_1})^*&\nu_1\cdot\nabla\mathcal{S}^0_{D_2}&\cdots&\nu_1\cdot\nabla\mathcal{S}^0_{D_{l_0}}\\
\nu_2\cdot\nabla\mathcal{S}^0_{D_1}&(\mathcal{K}^0_{D_2})^*&\cdots&\nu_2\cdot\nabla\mathcal{S}^0_{D_{l_0}}\\
\vdots&\vdots&\ddots&\vdots\\
\nu_{l_0}\cdot\nabla\mathcal{S}^0_{D_1}&\nu_{l_0}\cdot\nabla\mathcal{S}^0_{D_2}&\cdots&(\mathcal{K}^0_{D_{l_0}})^*
\end{pmatrix}.
\end{equation}

Define $l_0$-by-$l_0$ matrices type operators $\mathbb{M}_D$ and $\mathbb{N}_D$  on $\mathrm{TH(div},\partial D_1)\times\cdots\times \mathrm{TH(div},\partial D_{l_0})$ by
\begin{equation}\label{M_D}
\mathbb{M}_D:=(\mathcal{M}^0_{D_m,D_l})=\begin{pmatrix}
\mathcal{M}^0_{D_1}&\mathcal{M}^0_{D_1,D_2}&\cdots&\mathcal{M}^0_{D_1,D_{l_0}}\\
\mathcal{M}^0_{D_2,D_1}&\mathcal{M}^0_{D_2,D_2}&\cdots&\mathcal{M}^0_{D_2,D_{l_0}}\\
\vdots& \vdots&\ddots&\vdots\\
\mathcal{M}^0_{D_{l_0},D_1}&\mathcal{M}^0_{D_{l_0},D_2}&\cdots&\mathcal{M}^0_{D_{l_0}}\\
\end{pmatrix},
\end{equation}
where $l,m\in\{1,2,\dots,l_0\}$
and
\begin{equation}\label{N_D}
	\mathbb{N}_D:=\text{diag}\left((\mathcal{M}_{D_1,\Sigma},\mathcal{M}_{D_2,\Sigma},\dots, \mathcal{M}_{D_{l_0},\Sigma})(\lambda_\epsilon I+\mathcal{M}^0_\Sigma )^{-1}\right)
	\begin{pmatrix}
	\mathcal{M}^0_{\Sigma,D_1}&\mathcal{M}^0_{\Sigma,D_2}&\cdots&\mathcal{M}^0_{\Sigma,D_{l_0}}\\
	\mathcal{M}^0_{\Sigma,D_1}&\mathcal{M}^0_{\Sigma,D_2}&\cdots&\mathcal{M}^0_{\Sigma,D_{l_0}}\\
	\vdots& \vdots&\ddots&\vdots\\
	\mathcal{M}^0_{\Sigma,D_1}&\mathcal{M}^0_{\Sigma,D_2}&\cdots&\mathcal{M}^0_{\Sigma,D_{l_0}}\\
	\end{pmatrix}.
\end{equation}

\begin{thm}
	Let $(\mathbf{E}_s,\mathbf{H}_s)$ be the solution to \eqref{eq:MaterialDisturbution} and \eqref{frequencymaxwell}. Then for $\omega\in \mathbb{R}_+$ sufficiently small, there hold the following asymptotic expansions:
	\begin{equation}\label{asymptotic of H_s}
	\mathbf{H}_s=
	\begin{cases}
	\begin{aligned}
	&\hat{\mathbf{H}}_0-\epsilon_0\nabla\times\mathcal{A}^0_\Sigma[\Xi]+\nabla\mathcal{S}^0_\Sigma[\Theta]\\
	&\quad+\sum_{l'=1}^{l_0}\left(\epsilon_0\nabla\times\mathcal{A}^0_{D_{l'}}[\Psi^{(0)}_{l'}]-\mu_0\nabla\mathcal{S}^0_{D_{l'}}[\Pi_{l'}]\right)+O(\omega)&\mathrm{in}\ \mathbb{R}^3\setminus\overline{\Sigma},\\
	&\hat{\mathbf{H}}_0-\epsilon_s\nabla\times\mathcal{A}^0_\Sigma[\Xi]+\nabla\mathcal{S}^0_\Sigma[\Theta]\\
	&\quad+\sum_{l'=1}^{l_0}\left(\epsilon_s\nabla\times\mathcal{A}^0_{D_{l'}}[\Psi^{(0)}_{l'}]-\mu_0\nabla\mathcal{S}^0_{D_{l'}}[\Pi_{l'}]\right)+O(\omega)&\mathrm{in} \ \tilde{D},\\
	&-\gamma_l\nabla\times\mathcal{A}_\Sigma^0[\Xi]-\nabla\mathcal{S}^0_\Sigma[\Theta]\\
	&\quad+\sum_{l'=1}^{l_0}\left(\gamma_{l'}\nabla\times\mathcal{A}^0_{D_{l'}}[\Psi_{l'}^{(0)}]-\mu_0\nabla\mathcal{S}^0_{D_{l'}}[\Pi_{l'}]\right)+O(\omega)&\mathrm{in} \ D_l,
	\end{aligned}  	
	\end{cases}
	\end{equation}
	where $\Theta,\Xi\in\mathrm{TH(div,}\partial\Sigma)$ satisfy
	\begin{equation}
		\begin{aligned}
		\Theta&=\sum_{l{'}=1}^{l_0}(\lambda_\epsilon I+\mathcal{M}^0_\Sigma)^{-1}\mathcal{M}^0_{\Sigma,D_{l'}}[\Psi^{(0)}_{l'}],\\
		\Xi&=(\epsilon_s-\epsilon_0)\sum_{l{'}=1}^{l_0}\nu\cdot(\nabla\times\mathcal{A}^0_{D_{l'}}[\Psi_{l'}^{(0)}])-(\epsilon_s-\epsilon_0)\nu\cdot(\nabla\times\mathcal{A}^0_\Sigma[\Theta]),
		\end{aligned}
	\end{equation}
	and $\Psi^{(0)}_l\in\mathrm{TH(div,}\partial D_l),l=1,2,\dots,l_0$ satisfy
	\begin{equation}
		\Psi^{(0)}_l=\left((\mathbb{L}^\gamma_D)^{-1}[(\frac{(\nu_1\times\hat{\mathbf{H}}_0)^T}{\gamma_1-\epsilon_s},\frac{(\nu_2\times\hat{\mathbf{H}}_0)^T}{\gamma_2-\epsilon_s},\cdots,\frac{(\nu_{l_0}\times\hat{\mathbf{H}}_0)^T}{\gamma_{l_0}-\epsilon_s})^T]\right)\cdot(\vect{e}_l\otimes(1,1,1)^T),
	\end{equation}
	where $\mathbb{L}^\gamma_D$ is defined by
	\begin{equation}\label{L^gamma}
	\mathbb{L}^\gamma_D:=\begin{pmatrix}
	\lambda_{\gamma_1}I&0\cdots&0\\
	0&\lambda_{\gamma_2}I&\cdots&0\\
	\vdots&\vdots&\ddots&\vdots\\
	0&0&\cdots&\lambda_{\gamma_{l_0}}I
	\end{pmatrix}+\mathbb{M}_D-\mathbb{N}_D.
	\end{equation}
	with the operator $\mathbb{M}_D$ and $\mathbb{N}_D$ defined in \eqref{M_D} and \eqref{N_D}, respectively.
	Here the notation $\otimes$ stands for the Kronecker product and $\vect{e}_l$ is the $l_0$-dimensional Euclidean unit vector.  $\Pi_{l}\in L^2(\partial D_l), l=1,2,\dots,l_0$ are defined by
	\begin{equation}\label{Pi_l}
			\begin{aligned}
		\Pi_l=
		&\left((\mathbb{J}^\mu_D)^{-1}\left[\left(\frac{\nu_1\cdot\hat{\mathbf{H}}_0}{\mu_1-\mu_0},\frac{\nu_2\cdot\hat{\mathbf{H}}_0}{\mu_2-\mu_0},\cdots,\frac{\nu_{l_0}\cdot\hat{\mathbf{H}}_0}{\mu_{l_0}-\mu_0}\right)^T\right]\right)_l\\
		&-\left((\mathbb{J}^\mu_D)^{-1}\left[\left(\frac{\omega\gamma_1\mu_1\nu_1\cdot\mathbf{C}}{\mu_1-\mu_0},\frac{\omega\gamma_2\mu_2\nu_2\cdot\mathbf{C}}{\mu_2-\mu_0},\cdots,\frac{\omega\gamma_{l_0}\mu_{l_0}\nu_{l_0}\cdot\mathbf{C}}{\mu_{l_0}-\mu_0}\right)^T\right]\right)_l,
		\end{aligned}
	\end{equation}
	where $ \mathbb{J}^\mu_D$ is defined on $L^2(\partial D_1)\times L^2(\partial D_2)\times\cdots\times L^2(\partial D_{l_0})$ given by
	\begin{equation}\label{{J}^mu_D}
	\mathbb{J}^\mu_D:=\begin{pmatrix}
	\lambda_{\mu_1}I&0&\cdots&0\\
	0&\lambda_{\mu_2}I&\cdots&0\\
	\vdots&\vdots&\ddots&\vdots\\
	0&0&\cdots&\lambda_{\mu_{l_0}}I
	\end{pmatrix}-\mathbb{K}^*_D,
	\end{equation} with the operator $\mathbb{K}^*_D$ defined in \eqref{K^*_D} and $\mathbf{C}$ is defined by
	\begin{equation}\label{C}
		\mathbf{C}:=\nabla\times\mathcal{A}^0_\Sigma(\lambda_\epsilon I+\mathcal{M}^0_\Sigma)^{-1}\sum_{l'=1}^{l_0}\mathcal{M}^0_{\Sigma,D_{l'}}[\Psi^{(0)}_{l'}]-\nabla\times\sum_{l'=1}^{l_0}\mathcal{A}^0_{D_{l'}}[\Psi^{(0)}_{l'}].
	\end{equation}
	The parameter $\lambda_{\gamma_l}$, $\lambda_{\mu_l}$ and $\lambda_{\epsilon}$ are defined by
	\begin{equation}\label{lambda}
	\lambda_{\gamma_l}:=\frac{\gamma_l+\epsilon_s}{2(\gamma_l-\epsilon_s)},\quad\lambda_{\mu_l}:=\frac{\mu_l+\mu_0}{2(\mu_l-\mu_0)},\quad l=1,2,\dots,l_0,
	\end{equation}
	and
	\begin{equation}
		\lambda_\epsilon:=\frac{\epsilon_s-\epsilon_0}{2(\epsilon_s-\epsilon_0)}.
	\end{equation}
\end{thm}

\begin{proof}
The proof follows from that of Theorem 3.1 in \cite{deng2018identifying} with some straightforward modifications.
\end{proof}

\begin{lem}(Lemma 3.3 in \cite{deng2018identifying})\label{le:0101}
	$\hat{\mathbf{H}}_0$ introduced in \eqref{asymptotic of H_s} can be written as
	\begin{equation}\label{lem:hat H_0}
	\hat{\mathbf{H}}_0=\mathbf{H}_0+\sum_{l'=1}^{l_0}\nabla\mathcal{S}^0_{D_{l'}}[\varphi_{l'}]+O(\omega)\quad\mathrm{in}\ (\mathbb{R}^3\setminus\overline{\Sigma})\bigcup\tilde{D},
	\end{equation}
	where $\varphi_l,l=1,2,\dots,l_0$ satisfy
	\begin{equation}\label{varphi_l}
	\varphi_l=\left((\mathbb{J}^\epsilon_D)^{-1}[(\nu_1\cdot\mathbf{H}_0,\nu_2\cdot\mathbf{H}_0,\dots,\nu_{l_0}\cdot\mathbf{H}_0)^T]\right)_l,
	\end{equation}
	with the operator $\mathbb{J}^\epsilon_D$ defined in \eqref{{J}^mu_D} with $\lambda_{\mu_l}$ replaced by $\lambda_{\epsilon_l},l=1,2,\dots,l_0.$
\end{lem}

\subsection{Asymptotic expansions of $\mathbf{H}_s$ with respect to the anomaly size}

It is noted that if $\sigma_l,l=1,2,\dots,l_0$ are not identically zero, then the leading-order term in \eqref{asymptotic of H_s} may still depend on $\omega$. In the subsequent analysis, we make use of the leading-order term in the asymptotic low-frequency expansion in the representation formula \eqref{asymptotic of H_s} to further study the asymptotic expansion with respect to the size of the magnetic anomalies.

We first present the following lemma.
\begin{lem}\label{asy of K and M}
	Suppose $D_l,l=1,2,\dots,l_0$ are defined in \eqref{magnetized anomalies} with $\delta\in\mathbb{R}_+$ sufficiently small and the corresponding variation parameters $s_l=\delta^{\alpha_l},l=1,2,\dots,l_0$ with $\alpha_l>-1$.
	Then we have
	\begin{equation}\label{K*_D and M_D}
		\mathbb{K}^*_D=\mathbb{K}^0_\Omega+O((s_l\delta)^2),\quad\mathbb{M}_D=\mathbb{M}_\Omega+O((s_l\delta)^2),
	\end{equation}
	where $\mathbb{K}^*_\Omega,\mathbb{M}_\Omega$ are $l_0\times l_0$ matrix-valued operators defined by
	\begin{equation}
		\mathbb{K}^*_\Omega:=\text{diag}((\mathcal{K}^0_\Omega)^*,(\mathcal{K}^0_\Omega)^*,\cdots,(\mathcal{K}^0_\Omega)^*),\quad\mathbb{M}_\Omega:=\text{diag}(\mathcal{M}^0_\Omega,\mathcal{M}^0_\Omega,\cdots,\mathcal{M}^0_\Omega),
	\end{equation}
	respectively. $\mathbb{K}^*_D$ and $\mathbb{M}_D$ are defined in \eqref{K^*_D} and \eqref{M_D}, respectively.
\end{lem}
\begin{proof}
	We only prove the second assertion in \eqref{K*_D and M_D}, and the first one can be proved in a similar manner. For any $\mathbf{x},\mathbf{y}\in\partial D_l$, we let $\tilde{\mathbf{x}}=(s_l\delta)^{-1}(\mathbf{x}-\mathbf{z}_l),\tilde{\mathbf{y}}=(s_l\delta)^{-1}(\mathbf{y}-\mathbf{z}_l)\in\partial \Omega,l=1,2,\dots,l_0$. Define $\tilde{\Phi}(\tilde{\mathbf{y}})=\Phi(\mathbf{y})$. By using change of variables, one can show that there holds
	\begin{equation}\label{M^0_dl}
			\begin{aligned}
		\mathcal{M}^0_{D_l}[\Phi](\mathbf{x})&=\nu_{\mathbf{x}}\times\nabla_{\mathbf{x}}\times\int_{\partial {D_l}}\Gamma_0(\mathbf{x}-\mathbf{y})\Phi(\mathbf{y})ds_{\mathbf{y}}\\
		&=-\frac{1}{4\pi}\nu_{\mathbf{x}}\times\nabla_{\mathbf{x}}\times\int_{\partial\Omega}\frac{1}{\|\mathbf{x}-\mathbf{y}\|}\tilde{\Phi}(\tilde{\mathbf{y}})(s_l\delta)^2ds_{\tilde{\mathbf{y}}}\\
		&=-\frac{1}{4\pi}\nu_{\tilde{\mathbf{x}}}\times(s_l\delta)^{-1}\nabla_{\tilde{\mathbf{x}}}\times\int_{\partial\Omega}\frac{1}{(s_l\delta)\|\tilde{\mathbf{x}}-\tilde{\mathbf{y}}\|}\tilde{\Phi}(\tilde{\mathbf{y}})(s_l\delta)^2ds_{\tilde{\mathbf{y}}}\\
		&=\mathcal{M}^0_\Omega[\tilde{\Phi}](\mathbf{x}).
		\end{aligned}
	\end{equation}
On the other hand, letting $\mathbf{x}\in\partial D_m$ and $\mathbf{y}\in\partial D_l$ and $\tilde{\mathbf{x}}=(s_m\delta)^{-1}(\mathbf{x}-\mathbf{z}_m),\tilde{\mathbf{y}}=(s_l\delta)^{-1}(\mathbf{y}-\mathbf{z}_l)\in\partial \Omega$, where $l,m\in\{1,2,\dots,l_0\}$ and $l\neq m$, one can show that
	\begin{equation}
	\begin{aligned}
	\mathcal{M}^0_{D_m,D_l}[\Phi](\mathbf{x})&=\nu_{\mathbf{x}}\times\nabla_{\mathbf{x}}\times\mathcal{A}^0_{D_l}\Phi(\mathbf{x})\\
	&=\nu_{\mathbf{x}}\times\nabla_{\mathbf{x}}\times\int_{\partial {D_l}}\Gamma_0(\mathbf{x}-\mathbf{y})\Phi(\mathbf{y})ds_{\mathbf{y}}\\
	&=-\frac{1}{4\pi}\nu_{\mathbf{x}}\times\nabla_{\mathbf{x}}\times\int_{\partial\Omega}\frac{1}{\|\mathbf{x}-\mathbf{y}\|}\tilde{\Phi}(\tilde{\mathbf{y}})(s_l\delta)^2ds_{\tilde{\mathbf{y}}}\\
	&=-\frac{1}{4\pi}\nu_{\tilde{\mathbf{x}}}\times(s_m\delta)^{-1}\nabla_{\tilde{\mathbf{x}}}\times\int_{\partial\Omega}\frac{1}{\|\delta(s_m\tilde{\mathbf{x}}-s_l\tilde{\mathbf{y}})+(\mathbf{z}_m-\mathbf{z}_l)\|}\tilde{\Phi}(\tilde{\mathbf{y}})(s_l\delta)^2ds_{\tilde{\mathbf{y}}}
	\end{aligned}
	\end{equation}
By Taylor expansion there holds
	\begin{equation}
		\begin{aligned}
		&\|\delta(s_m\tilde{\mathbf{x}}-s_l\tilde{\mathbf{y}})+(\mathbf{z}_m-\mathbf{z}_l)\|^{-1}\\
		=&\|\mathbf{z}_m-\mathbf{z}_l\|^{-1}-\delta\frac{\la s_m\tilde{\mathbf{x}}-s_l\tilde{\mathbf{y}},\bZ_m-\bZ_l\ra}{ \|\mathbf{z}_m-\mathbf{z}_l\|^{3}} -\frac{3}{2}\delta^2\Big(\frac{\|s_m\tilde{\mathbf{x}}-s_l\tilde{\mathbf{y}}\|^2}{\|\mathbf{z}_m-\mathbf{z}_l\|^{3}}
-\frac{\la\bZ_m-\bZ_l,s_m\tilde{\mathbf{x}}-s_l\tilde{\mathbf{y}}\ra}{\|\mathbf{z}_m-\mathbf{z}_l\|^{5}}\Big)+\cdots\\
=&\|\mathbf{z}_m-\mathbf{z}_l\|^{-1}+O(s_m\delta+s_l\delta)=\|\mathbf{z}_m-\mathbf{z}_l\|^{-1}+O(\delta^{1+\alpha_m}+\delta^{1+\alpha_l})\\
=&\|\mathbf{z}_m-\mathbf{z}_l\|^{-1}+\begin{cases}
		O(s_l\delta)&\text{if }\alpha_m>\alpha_l,\\
		O(s_m\delta)&\text{if }\alpha_m=\alpha_l,\\
		O(s_m\delta)&\text{if }\alpha_m<\alpha_l.
		\end{cases}
		\end{aligned}
	\end{equation}
Thus, for the case when $\alpha_m>\alpha_l$,
	\begin{equation}
	\begin{aligned}
		\mathcal{M}^0_{D_m,D_l}[\Phi](\mathbf{x})&=O((s_m\delta)^{-1}\cdot (s_l\delta)\cdot(s_l\delta)^2)=O(\delta^{3\alpha_l-\alpha_m}\cdot\delta^2)=O((s_l\delta)^2).
	\end{aligned}
	\end{equation}
	By direct calculations, we can also verify the cases when $\alpha_m\leq \alpha_l$. Hence,
	\begin{equation}\label{M^0_lm}
		\mathcal{M}^0_{D_m,D_l}[\Phi](\mathbf{x})=O((s_l\delta)^2).
	\end{equation}
	By substituting \eqref{M^0_dl} and \eqref{M^0_lm} back into \eqref{M_D}, one readily has \eqref{K*_D and M_D}.
	
	The proof is complete.
\end{proof}

We introduce some notations for further analysis. In what follows, we let $\mathbf{H}^{(0)}_0$ be the leading-order term of $\mathbf{H}_0$ in \eqref{asymptotic of H_s}, $\hat{\mathbf{H}}^{(0)}_0$ be the leading-order term of $\hat{\mathbf{H}}_0$ (and $\varphi_l^{(0)}$ is the leading-order term of $\varphi_l,l=1,2,\dots,l_0$). Similarly, we let $\mathbf{H}^{(0)}$ and $\mathbf{H}^{(0)}_s$ be the leading-order term of $\mathbf{H}$ and $\mathbf{H}_s$, respectively.
Recall the following important result: .
\begin{lem}\label{lem:citerelationH_0}{(Lemma 3.6 in \cite{deng2018identifying})}
	For any simply connected domain $D_l$ and a conservative gradient field $\mathbf{A}$ in $\mathbb{R}^3$, three holds the following relation:
	\begin{equation}
	\begin{aligned}
	&\frac{1}{\gamma_l-\epsilon_s}\nabla\times\mathcal{A}^0_{D_l}(\lambda_{\gamma_l}I+\mathcal{M}^0_{D_l} )^{-1}[\nu_l\times\mathbf{A}]\\
	&\quad=\begin{cases}
	\frac{1}{\gamma_l-\epsilon_s}\nabla\mathcal{S}^0_{D_l}(\lambda_{\gamma_l}I+(\mathcal{K}^0_{D_l})^*)^{-1}[\nu_l\cdot\mathbf{A}]&\mathrm{in}\ \mathbb{R}^3\setminus\overline{D_l},\\
	\frac{1}{\gamma_l}\mathbf{A}+\frac{\epsilon_s}{\gamma_l(\gamma_l-\epsilon_s)}\nabla\mathcal{S}^0_{D_l}(\lambda_{\gamma_l}I+(\mathcal{K}^0_{D_l})^*)^{-1}[\nu_l\cdot\mathbf{A}]&\mathrm{in}\ D_l.
	\end{cases}
	\end{aligned}
	\end{equation}
\end{lem}
By \cite{deng2018identifying} one can find that the leading-order term $\mathbf{H}^{(0)}_0$ is a conservative gradient field.
With the above preparations, we are in a position to derive the first main result, i.e., the asymptotic expansion of the geomagnetic field $\mathbf{H}_s$ with respect to the size of the magnetic anomalies.

\begin{thm}\label{th:main01}
	Suppose $D_l, l=1,2,\dots,l_0$ are defined in \eqref{magnetized anomalies} with $\delta\in\mathbb{R}_+$ sufficiently small and $s_l=\delta^{\alpha_l},l=1,2,\dots,l_0$ with $\alpha_l>-1$. Let $(\mathbf{E}_s,\mathbf{H}_s)$ be the solution to \eqref{eq:MaterialDisturbution} and \eqref{frequencymaxwell}. The for $\mathbf{x}\in \mathbb{R}^3\setminus\overline{\Sigma}$, there holds the following asymptotic expansion,
	\begin{equation}\label{asy expansion of H_s Size}
		\begin{aligned}
		\mathbf{H}_s^{(0)}(\mathbf{x})=&\mathbf{H}^{(0)}_0(\mathbf{x})-\delta^3\sum_{l=1}^{l_0}\delta^{3\alpha_{l}}\nabla\left(\nabla\Gamma_0(\mathbf{x}-\mathbf{z}_l)^T\mathbf{P}_0\mathbf{H}_0^{(0)}(\mathbf{z}_l)\right)\\
		&-\delta^3\sum_{l=1}^{l_0}\delta^{3\alpha_{l}}\Big(\epsilon_0\nabla\left(\nabla\Gamma_0(\mathbf{x}-\mathbf{z}_l)^T\mathbf{D}_l\mathbf{H}_0^{(0)}(\mathbf{z}_l)\right)\\
		&\quad\quad\quad\quad\quad\quad-\mu_0\nabla\left(\nabla\Gamma_0(\mathbf{x}-\mathbf{z}_l)^T\mathbf{M}_l\mathbf{H}_0^{(0)}(\mathbf{z}_l)\right)\Big)+O(\sum_{l=1}^{l_0}\delta^{4\alpha_l+4}),
		\end{aligned}
	\end{equation}
	where $\mathbf{P}_0$ is defined by
	\begin{equation}\label{polar01}
		\mathbf{P}_0:=\int_{\partial \Omega}\tilde{\mathbf{y}}\left(\lambda_\epsilon I -(\mathcal{K}^0_{\Omega})^*\right)^{-1}[\nu_l]ds_{\tilde{\mathbf{y}}}.
	\end{equation}
	The polarization tensors $\mathbf{D}_l$ and $ \vect{M_l}$ are $3\times3$ matrices defined by
	\begin{equation}\label{polarD}
	\mathbf{D}_l=\frac{1}{\gamma_l-\epsilon_s}\frac{\epsilon_s}{\epsilon_s-\epsilon_0}\int_{\partial \Omega}\tilde{\mathbf{y}}\left(\lambda_{\gamma_l}I+(\mathcal{K}^0_\Omega)^*\right)^{-1}\left(\lambda_\epsilon I -(\mathcal{K}^0_{\Omega})^*\right)^{-1}[\nu_l]ds_{\tilde{\mathbf{y}}},
	\end{equation}
	and
    \begin{equation}\label{polarM}
	\mathbf{M}_l=\frac{1}{\mu_l-\mu_0}\frac{\epsilon_s}{\epsilon_s-\epsilon_0}\int_{\partial \Omega}\tilde{\mathbf{y}}\left(\lambda_{\mu_l}I-(\mathcal{K}^0_\Omega)^*\right)^{-1}\left(\lambda_\epsilon I -(\mathcal{K}^0_{\Omega})^*\right)^{-1}[\nu_l]ds_{\tilde{\mathbf{y}}},
	\end{equation}
	respectively, $l=1,2,\dots,l_0$. More specifically, setting $\mathbf{M}_l=((\mathbf{D}_l)_{mn}),m,n=1,2,3,\tilde{\mathbf{y}}=(\tilde{\mathbf{y}}_1,\tilde{\mathbf{y}}_2,\tilde{\mathbf{y}}_3)^T$ and $\nu_l=(\nu_l^{(1)},\nu_l^{(2)},\nu_l^{(3)})^T$, we have
	\begin{equation}
		 (\mathbf{D}_l)_{mn}=\frac{1}{\gamma_l+\epsilon_s}\frac{\epsilon_s}{\epsilon_s-\epsilon_0}\int_{\partial \Omega}\tilde{\mathbf{y}}_m\left(\lambda_{\gamma_l}I-(\mathcal{K}^0_\Omega)^*\right)^{-1}\left(\lambda_\epsilon I -(\mathcal{K}^0_{\Omega})^*\right)^{-1}[\nu_l^{(n)}]ds_{\tilde{\mathbf{y}}}.
	\end{equation}
$\mathbf{P}_0$ and $\mathbf{M}_l$ have similar forms.
\end{thm}
\begin{proof}
	
	We note that either $\omega\gamma_l$ or $1/(\gamma_l-\epsilon_l)$ is of order $\omega,l=1,2,\dots,l_0$, no matter $\sigma$ is zero or nonzero. One can immediately find that the second term in \eqref{Pi_l} is of order $\omega$. By \eqref{asymptotic of H_s}, it then can be seen that the leading order term $\mathbf{H}^{(0)}_s$ has the following form.
		\begin{equation}
		\begin{aligned}
		\mathbf{H}^{(0)}_s=&\hat{\mathbf{H}}^{(0)}_0-\epsilon_0\sum_{l=1}^{l_0}\nabla\times\mathcal{A}^0_\Sigma(\lambda_\epsilon I+\mathcal{M}^0_\Sigma)^{-1}\nu\times\nabla\times\mathcal{A}^0_{D_l}[\Psi'_l]\\
		&+(\epsilon_s-\epsilon_0)\sum_{l=1}^{l_0}\nabla\mathcal{S}^0_{\Sigma}[\nu\cdot\left(\nabla\times\mathcal{A}^0_{D_l}[\Psi'_l]\right)]\\
		&-(\epsilon_s-\epsilon_0)\sum_{l=1}^{l_0}\nabla\mathcal{S}^0_\Sigma\nu\cdot\nabla\times\mathcal{A}^0_\Sigma(\lambda_{\epsilon} I+\mathcal{M}^0_\Sigma)^{-1}\nu\times\nabla\times\mathcal{A}^0_{D_l}[\Psi'_l]\\
		&+\sum_{l=1}^{l_0}\left(\epsilon_0\nabla\times\mathcal{A}^0_{D_l}[\Psi'_l]-\mu_0\nabla\mathcal{S}^0_{D_l}[\Pi'_l]\right)\quad\text{in }\mathbb{R}^3\setminus{\overline{\Sigma}},
		\end{aligned}
		\end{equation}
		where $\Psi'_l$ and $\Pi'_l$ are defined by
		\begin{equation}\label{Psi'_l}
		\Psi'_l=\left((\mathbb{L}_D^\gamma)^{-1}\left[\left(\frac{(\nu_1\times\hat{\mathbf{H}}^{(0)}_0)^T}{\gamma_1-\epsilon_s},\frac{(\nu_2\times\hat{\mathbf{H}}^{(0)}_0)^T}{\gamma_2-\epsilon_s},\cdots,\frac{(\nu_{l_0}\times\hat{\mathbf{H}}^{(0)}_0)^T}{\gamma_{l_0}-\epsilon_s}\right)^T\right]\right)\cdot(\vect{e}_l\otimes(1,1,1)^T),
		\end{equation}
		and
		\begin{equation}\label{Pi'_l}
		\Pi'_l=\left((\mathbb{J}^\mu_D)^{-1}\left[\left(\frac{\nu_1\cdot\hat{\mathbf{H}}_0^{(0)}}{\mu_1-\mu_0},\frac{\nu_2\cdot\hat{\mathbf{H}}_0^{(0)}}{\mu_2-\mu_0},\cdots,\frac{\nu_{l_0}\cdot\hat{\mathbf{H}}_0^{(0)}}{\mu_{l_0}-\mu_0}\right)^T\right]\right)_l
		\end{equation}
		respectively, $l=1,2,\dots,l_0$. It can be verified that
	\begin{equation}\label{dervigence free}
		\nabla_{\partial D_l}\cdot\Psi'_l=0,\quad l=1,2,\dots l_0.
	\end{equation}
	Thus, $\nabla\times\mathcal{A}^0_{D_l}[\Psi'_l]$ is a gradient field of harmonic function in $\mathbb{R}^3\setminus\overline{D_l}$. By using the relations $$\nabla\times\mathcal{A}^0_{D_{l'}}[\nu_{l'}\times\hat{\mathbf{H}}_0]=\nabla\mathcal{S}^0_{D_l}[\nu_l\cdot\hat{\mathbf{H}}_0]+O(\omega)$$ and $$\nabla\times\mathcal{A}^0_\Sigma(\lambda_\epsilon I+\mathcal{M}^0_\Sigma)^{-1}\mathcal{M}^0_{\Sigma,D_{l'}}[\nu_{l'}\times\hat{\mathbf{H}}_0]=\nabla\mathcal{S}^0_\Sigma(\lambda_\epsilon I+(\mathcal{K}^0_\Sigma)^*)^{-1}\nu\cdot\nabla\mathcal{S}^0_{D_{l'}}[\nu_{l'}\cdot\hat{\mathbf{H}}_0]+O(\omega)$$ in $\mathbb{R}^3\setminus\overline{\Sigma}$, (see Lemma 3.2 in \cite{deng2018identifying}) and \eqref{dervigence free}, one can derive that
	
	\begin{equation}\label{simpler H^{0}_s}
		\mathbf{H}^{(0)}_s=\hat{\mathbf{H}}^{(0)}_0+\sum_{l=1}^{l_0}\left(\epsilon_0\nabla\times\mathcal{A}^0_{D_l}[\Psi'_l]-\mu_0\nabla\mathcal{S}^0_{D_l}[\Pi'_l]\right)\quad\text{in }\mathbb{R}^3\setminus\overline{\Sigma}.
	\end{equation}
	As before, for $\mathbf{y}\in \partial D_l$, we let $\tilde{\mathbf{y}}=(s_l\delta)^{-1}(\mathbf{y}-\mathbf{z}_l)\in\partial\Omega$, and define $\tilde{\Psi}'_l:=\Psi'_l(\mathbf{y}),\tilde{\Pi}'_l:=\Pi'_l(\mathbf{y}),l=1,2,\dots,l_0$ and $\tilde{\hat{\mathbf{H}}}^{(0)}_0(\tilde{\mathbf{y}}):=\hat{\mathbf{H}}^{(0)}_0(\mathbf{y}).$ Then by Lemma \ref{asy of K and M}, one has
	\begin{equation}\label{asy of tildePsi}
		\tilde{\Psi}'_l(\tilde{\mathbf{y}})=\frac{1}{\gamma_l-\epsilon_s}\left(\lambda_{\gamma_l}I+\mathcal{M}^0_\Omega\right)^{-1}\left[\nu_l\times\tilde{\hat{\mathbf{H}}}_0^{(0)}\right](\tilde{\mathbf{y}})+O((s_l\delta)^2),
	\end{equation}
	and
	
	\begin{equation}\label{asy of tildePi}
	\tilde{\Pi}'_l(\tilde{\mathbf{y}})=\frac{1}{\mu_l-\mu_0}\left(\lambda_{\mu_l}I-(\mathcal{K}^0_\Omega)^*\right)^{-1}\left[\nu_l\cdot\tilde{\hat{\mathbf{H}}}_0^{(0)}\right](\tilde{\mathbf{y}})+O((s_l\delta)^2),
	\end{equation}
	$l=1,2,\dots,l_0.$ Hence, by using Lemma \ref{lem:citerelationH_0}, there holds
	
	\begin{equation}\label{nablatimesA}
		\begin{aligned}
		&\nabla\times\mathcal{A}^0_{D_l}[\Psi'_l]\\
		&=\nabla\times\mathcal{A}^0_{D_l}\left((\mathbb{L}_D^\gamma)^{-1}\left[\left(\frac{(\nu_1\times\hat{\mathbf{H}}^{(0)}_0)^T}{\gamma_1-\epsilon_s},\frac{(\nu_2\times\hat{\mathbf{H}}^{(0)}_0)^T}{\gamma_2-\epsilon_s},\cdots,\frac{(\nu_{l_0}\times\hat{\mathbf{H}}^{(0)}_0)^T}{\gamma_{l_0}-\epsilon_s}\right)^T\right]\right)\cdot(\vect{e}_l\otimes(1,1,1)^T)\\
		&=\frac{1}{\gamma_l-\epsilon_s}\nabla\times\mathcal{A}^0_{D_l}\left(\lambda_{\gamma_1}I+\mathcal{M}^0_{D_l}\right)^{-1}\left[\nu_l\times{\hat{\mathbf{H}}}^{(0)}_0\right]({\mathbf{y}})\\
		&=\frac{1}{\gamma_l-\epsilon_s}\nabla\times\mathcal{A}^0_{D_l}\left(\lambda_{\gamma_1}I+\mathcal{M}^0_\Omega\right)^{-1}\left[\nu_l\times\tilde{\hat{\mathbf{H}}}^{(0)}_0\right](\tilde{\mathbf{y}})+O((s_l\delta)^4)\\		
		&=\frac{1}{\gamma_l-\epsilon_s}\nabla\mathcal{S}^0_{D_l}\left[(\lambda_{\gamma_l}I+(\mathcal{K}^0_\Omega)^*)^{-1}[\nu_l\cdot\tilde{\hat{\mathbf{H}}}^{(0)}_0(\tilde{\mathbf{y}})]\right]+O((s_l\delta)^4)\\
		&:=\nabla\mathcal{S}^0_{D_l}[\tilde{Q_l}]+O((s_l\delta)^4)\quad\text{in }\mathbb{R}^3\setminus\overline{\Sigma}.
		\end{aligned}
	\end{equation}
	On the other hand, by the Taylor expansion, there holds
	
	\begin{equation}\label{Taylor}
		\mathbf{H}^{(0)}_0(\mathbf{y})=\mathbf{H}^{(0)}_0(\mathbf{z}_l)+s_l\delta\nabla\mathbf{H}^{(0)}_0(\tilde{\mathbf{y}})+O((s_l\delta)^2),
	\end{equation}
	and so by using \eqref{lem:hat H_0}, \eqref{Taylor} and \eqref{trace formula}, one has
	
	\begin{equation}\label{nudotH}
	\begin{aligned}
	   \nu_l\cdot\tilde{\hat{\mathbf{H}}}^{(0)}_0(\tilde{\mathbf{y}})
	   &=\nu_l\cdot\hat{\mathbf{H}}^{(0)}_0(\mathbf{y})=\nu_l\cdot\mathbf{H}^{(0)}_0(\mathbf{y})+\nu_l\cdot\sum_{l=1}^{l_0}\nabla\mathcal{S}_{D_{l}}^0[{\varphi}^{(0)}_{l}](\mathbf{y})\\
	   &=\nu_l\cdot\left(\mathbf{H}^{(0)}_0(\mathbf{z}_l)+s_l\delta\nabla\mathbf{H}^{(0)}_0(\tilde{\mathbf{y}})+O((s_l\delta)^2)\right)+\left(\frac{I}{2}+(\mathcal{K}^0_{D_l})^*\right)[\tilde{\varphi}_l^{(0)}(\tilde{\mathbf{y}})]\\
	   &=\nu_l\cdot\mathbf{H}^{(0)}_0(\mathbf{z}_l)+O(s_l\delta)+\left(\frac{I}{2}+(\mathcal{K}^0_{\Omega})^*\right)[\tilde{\varphi}_l^{(0)}(\tilde{\mathbf{y}})]+O((s_l\delta)^2)\\
	   &=\nu_l\cdot\mathbf{H}^{(0)}_0(\mathbf{z}_l)+\left(\frac{I}{2}+(\mathcal{K}^0_{\Omega})^*\right)[\tilde{\varphi}_l^{(0)}(\tilde{\mathbf{y}})]+O(s_l\delta),
	\end{aligned}	
	\end{equation}
where $\tilde{\varphi}_l^{(0)}(\tilde{\mathbf{y}}):=\varphi_l^{(0)}(\mathbf{y})$ and $\varphi_l(\mathbf{y})=\varphi_l^{(0)}(\mathbf{y})+O(\omega)$. By \eqref{varphi_l} and \eqref{Taylor} one has
	\begin{equation}\label{tilde{varphi}_l^{(0)}}
	\begin{aligned}
	     \tilde{\varphi}_l^{(0)}(\tilde{\mathbf{y}})&=\left(\lambda_\epsilon I-(\mathcal{K}^0_\Omega)^*\right)^{-1}[\nu_l\cdot\mathbf{H}^{(0)}_0]+O((s_l\delta)^2)\\
	     &=\left(\lambda_\epsilon I-(\mathcal{K}^0_\Omega)^*\right)^{-1}[\nu_l\cdot\mathbf{H}^{(0)}_0(\mathbf{z}_l)]+O(s_l\delta).
	\end{aligned}	
	\end{equation}
Thus, by combining \eqref{nudotH}, \eqref{tilde{varphi}_l^{(0)}} and \eqref{nudotH} one can derive that
	\begin{equation}\label{finalnuH^{(0)}_0(z)}
		\begin{aligned}
		    \nu_l\cdot\tilde{\hat{\mathbf{H}}}^{(0)}_0(\tilde{\mathbf{y}})
		    &=\nu_l\cdot\mathbf{H}^{(0)}_0(\mathbf{z}_l)+\left(\frac{I}{2}+(\mathcal{K}^0_{\Omega})^*\right)\left(\lambda_\epsilon I-(\mathcal{K}^0_\Omega)^*\right)^{-1}[\nu_l\cdot\mathbf{H}^{(0)}_0(\mathbf{z}_l)]+O(s_l\delta)\\
		    &=\left[\left(\lambda_\epsilon I-(\mathcal{K}^0_\Omega)^*\right)+\left(\frac{I}{2}+(\mathcal{K}^0_{\Omega})^*\right)\right]\left(\lambda_\epsilon I-(\mathcal{K}^0_\Omega)^*\right)^{-1}[\nu_l\cdot\mathbf{H}^{(0)}_0(\mathbf{z}_l)]+O(s_l\delta)\\
		    &=\frac{\epsilon_s}{\epsilon_s-\epsilon_0}I\left(\lambda_\epsilon I-(\mathcal{K}^0_\Omega)^*\right)^{-1}[\nu_l\cdot\mathbf{H}^{(0)}_0(\mathbf{z}_l)]+O(s_l\delta).\\
		\end{aligned}
	\end{equation}
	For $\mathbf{x}-\mathbf{z}_l\gg s_l\delta$ there holds
	\begin{equation}\label{tayor of green}
		\Gamma_0(\mathbf{x}-\mathbf{y})=\Gamma_0(\mathbf{x}-\mathbf{z}_l)-s_l\delta\nabla\Gamma_0(\mathbf{x}-\mathbf{z}_l)^T\tilde{\mathbf{y}}+O((s_l\delta)^2).
	\end{equation}
	Define $\tilde{Q}_l(\tilde{\mathbf{y}}):=Q_l(\mathbf{y})$, where $\tilde{Q}_l$ is given in \eqref{nablatimesA}. By using change of variables and \eqref{lem:hat H_0}, one has
	\begin{equation}\label{H^{(0)}_swith polarization}
	\begin{aligned}
	\mathbf{H}^{(0)}_s(\mathbf{x})
	&={\mathbf{H}}^{(0)}_0(\mathbf{x})+\sum_{l=1}^{l_0}\nabla\mathcal{S}^0_{D_{l}}[\varphi_l^{(0)}]\\
	&\quad+\sum_{l=1}^{l_0}\Big(\epsilon_0\frac{1}{\gamma_l-\epsilon_s}\nabla\mathcal{S}^0_{D_l}\left[(\lambda_{\gamma_l}I+(\mathcal{K}_\Omega)^*)^{-1}[\nu_l\cdot\tilde{\hat{\mathbf{H}}}^{(0)}_0](\tilde{\mathbf{y}})\right]+O((s_l\delta)^4)\\
	&\quad-\mu_0\frac{1}{\mu_l-\mu_0}\nabla\mathcal{S}^0_{D_l}\left[\left(\lambda_{\mu_l}I-(\mathcal{K}^0_\Omega)^*\right)^{-1}[\nu_l\cdot\tilde{\hat{\mathbf{H}}}_0^{(0)}](\tilde{\mathbf{y}})+O((s_l\delta)^2)\right]\Big)\\
	&={\mathbf{H}}^{(0)}_0(\mathbf{x})+\sum_{l=1}^{l_0}\int_{\partial D_l}\nabla\Gamma_0(\mathbf{x}-\mathbf{y})[ \tilde{\varphi}_l^{(0)}(\tilde{\mathbf{y}})]ds_{\mathbf{y}}\\
	&\quad+\epsilon_0\sum_{l=1}^{l_0}\frac{1}{\gamma_l-\epsilon_s}\int_{\partial D_l} \nabla\Gamma_0(\mathbf{x}-\mathbf{y})\left[(\lambda_{\gamma_l}I+(\mathcal{K}_\Omega)^*)^{-1}[\nu_l\cdot\tilde{\hat{\mathbf{H}}}^{(0)}_0](\tilde{\mathbf{y}})\right]ds_{\mathbf{y}}\\
	&\quad-\mu_0\sum_{l=1}^{l_0}\frac{1}{\mu_l-\mu_0}\int_{\partial D_l}\nabla\Gamma_0(\mathbf{x}-\mathbf{y})\left[(\lambda_{\mu_l}I+(\mathcal{K}_\Omega)^*)^{-1}[\nu_l\cdot\tilde{\hat{\mathbf{H}}}^{(0)}_0](\tilde{\mathbf{y}})\right]ds_{\mathbf{y}}\\
	&\quad+O(\sum_{l=1}^{l_0}(s_l\delta)^4).
	\end{aligned}
	\end{equation}
Finally, by using the definitions \eqnref{polar01}-\eqnref{polarM} and substituting \eqref{asy of tildePsi}-\eqref{tayor of green} into \eqref{simpler H^{0}_s}, one obtains
\begin{equation}
\begin{split}
\mathbf{H}^{(0)}_s(\mathbf{x})=&{\mathbf{H}}^{(0)}_0(\mathbf{x})-\delta^3\sum_{l=1}^{l_0}s_l^3\nabla^2\Gamma_0(\mathbf{x}-\mathbf{z}_l)\mathbf{P}_0{{\mathbf{H}}}^{(0)}_0(\mathbf{z}_l)\\
&\quad-\delta^3\sum_{l=1}^{l_0}s_l^3\nabla^2\Gamma_0(\mathbf{x}-\mathbf{z}_l)\mathbf{D}_l{{\mathbf{H}}}^{(0)}_0(\mathbf{z}_l)\\
	&\quad+\delta^3\sum_{l=1}^{l_0}s_l^3\nabla^2\Gamma_0(\mathbf{x}-\mathbf{z}_l)\mathbf{M}_l{{\mathbf{H}}}^{(0)}_0(\mathbf{z}_l)+O(\sum_{l=1}^{l_0}(s_l\delta)^4).
\end{split}
\end{equation}
The first equatily of \eqref{H^{(0)}_swith polarization} is obtain by using the following facts:	
	\begin{equation}\label{int Q int Pi}
		\int_{\partial\Omega}\tilde{Q}_l=O((s_l\delta)^2),\quad\int_{\partial\Omega}\tilde{\Pi}'_l=O((s_l\delta)^2).
	\end{equation}
	To prove \eqref{int Q int Pi}, we first set
	\begin{equation}
		\phi_l(\tilde{\mathbf{y}}):=\left(\lambda_{\gamma_l}I+(\mathcal{K}^0_\Omega)^*\right)^{-1}[\nu_l\cdot\tilde{\hat{\mathbf{H}}}^{(0)}_0](\tilde{\mathbf{y}}).
	\end{equation}
	By using the jump formula \eqref{trace formula} and integration by parts, one can show that there holds
	\begin{equation}
	\begin{aligned}
		0&=\int_{\partial \Omega}\nu_l\cdot\tilde{\hat{\mathbf{H}}}^{(0)}_0=\int_{\partial \Omega}\left(\lambda_{\gamma_l}I+(\mathcal{K}^0_\Omega)^*\right)[\phi_l]\\
		&=(\lambda_{\gamma_l}+1/2)\int_{\partial\Omega}\phi_l-\int_{\partial \Omega}\left(1/2I+(\mathcal{K}^0_\Omega)^*\right)[\phi_l]\\
		&=(\lambda_{\gamma_l}+1/2)\int_{\partial\Omega}\phi_l-\int_{\partial \Omega}\nu\cdot\nabla\mathcal{S}_\Omega^0[\phi_l]\vert_-=(\lambda_{\gamma_l}+1/2)\int_{\partial\Omega}\phi_l.
	\end{aligned}
	\end{equation}
	We only prove the first assertion in \eqref{int Q int Pi}, and the second one can be proved in a similar manner.
	The proof is complete.
\end{proof}

For notational convenience, in the sequel, we introduce the matrix $\mathbf{P}_l$ by
\begin{equation}\label{polarization}
	\mathbf{P}_l:=\mu_0\mathbf{M}_l-\epsilon_0\mathbf{D}_l-\mathbf{P}_0,
\end{equation}
where $\mathbf{M}_l$ and $\mathbf{D}_l$ are defined in \eqref{polarM} and \eqref{polarD}, respectively.

\begin{rem}
	We remark that the matrix $\mathbf{P}_l$ is exactly same as the matrix obtained in \cite{deng2018identifying}. Hence, it satisfies the following nonsingularity properties: \end{rem}
\begin{lem}[Lemma 3.7 in \cite{deng2018identifying}]\label{lem:nonsingularI}
	If $\sigma_l\neq0,l=1,2,\dots,l_0$ and $\epsilon_s=\epsilon_0$, then $\mathbf{P}_l=\mu_0\mathbf{M}_l+O(\omega)$ is nonsingular.
\end{lem}

\begin{lem}[Lemma 3.8 in \cite{deng2018identifying}]\label{lem:nonsingularII}
	Suppose $\Omega$ is a ball. Let $\mathbf{P}_l$ be defined in \eqref{polarization}. If there holds
	\begin{equation}
		\mu_l\epsilon_s^2+2(\mu_0-\mu_l)\epsilon_s\gamma_l+2(\mu_l+2\mu_0)\epsilon_0\gamma_l\neq\mu_0\epsilon_s^2,
	\end{equation}
	then $\mathbf{P}_l$ is nonsingular.
\end{lem}
Let $\mathbf{H}_{s}^{(0)}$ be the leading term of $\mathbf{H}_{s}$. Then from \eqref{asy expansion of H_s Size}, there holds, for $\mathbf{x}\in  \mathbb{R}^3\setminus\overline{\Sigma},$
\begin{equation}\label{eq:asy of H_sj in term of polar}	\mathbf{H}_{s}^{(0)}(\mathbf{x})=\mathbf{H}_0^{(0)}(\mathbf{x})+\delta^3\sum_{l=1}^{l_0}\delta^{3\alpha_l}
\left(\nabla\left(\nabla\Gamma_0(\mathbf{x}-\mathbf{z}_l)\right)^T\mathbf{P}_l\mathbf{H}^{(0)}_0(\mathbf{z}_l)\right)+O(\sum_{l=1}^{l_0}\delta^{4\alpha_l+4}).
\end{equation}
On the other hand, we let $\mathbf{H}^{(0)}$ be the leading term of $\mathbf{H}$, with variation parameters $\alpha_l=0,l=1,2,\dots,l_0$. Then there holds, for $\mathbf{x}\in \mathbb{R}^3\setminus\overline{\Sigma},$
\begin{equation}\label{eq:asy of H_orij in term of polar}
\mathbf{H}^{(0)}(\mathbf{x})=\mathbf{H}_0^{(0)}(\mathbf{x})+\delta^3\sum_{l=1}^{l_0}\left(\nabla\left(\nabla\Gamma_0(\mathbf{x}-\mathbf{z}_l)\right)^T\mathbf{P}_l\mathbf{H}^{(0)}_0(\mathbf{z}_l)\right)+O(\delta^{4}).
\end{equation}
By combining \eqnref{eq:asy of H_sj in term of polar} and \eqnref{eq:asy of H_orij in term of polar} one thus has the following result:
\begin{thm}\label{th:main02}
Suppose $D_l, l=1,2,\dots,l_0$ are defined in \eqref{magnetized anomalies} with $\delta\in\mathbb{R}_+$ sufficiently small and $s_l=\delta^{\alpha_l},l=1,2,\dots,l_0$ with $-1/4<\alpha_l< 1/3$. Let $(\mathbf{E}_s,\mathbf{H}_s)$ be the solution to \eqref{eq:MaterialDisturbution} and \eqref{frequencymaxwell}. Then for $\mathbf{x}\in \mathbb{R}^3\setminus\overline{\Sigma}$, there holds the following asymptotic expansion,
	\begin{equation}\label{asy expansion of H_s Size}
		\begin{aligned}
\mathbf{H}_s^{(0)}(\mathbf{x})=&\mathbf{H}^{(0)}(\mathbf{x})-\delta^3\sum_{l=1}^{l_0}(\delta^{3\alpha_{l}}-1)\nabla\left(\nabla\Gamma_0(\mathbf{x}-\mathbf{z}_l)^T\mathbf{P}_l\mathbf{H}_0^{(0)}(\mathbf{z}_l)\right)\\
		&+O(\max(\delta^{4\alpha_1+4}, \delta^{4\alpha_2+4}, \ldots, \delta^{4\alpha_{l_0}+4}, \delta^4)).
		\end{aligned}
	\end{equation}
\end{thm}
\begin{rem}\label{rem:ll1}
Two remarks are in order. First, the asymptotic formula \eqnref{asy expansion of H_s Size} describes the leading term of the difference of the magnetic field before and after the growth ($-1/4<\alpha_l<0$) or shrinkage ($0<\alpha_l<1/3$) of the magnetic anomalies, $l=1, 2, \ldots, l_0$. The terms $\mathbf{P}_l\mathbf{H}_0^{(0)}(\mathbf{z}_l)$ are not possible to calculate or measure directly, and we use the measurement of perturbation of magnetic fields to estimate the terms in the subsequent section. Second, the restriction $-1/4<\alpha_l< 1/3$ is to ensure that $3(1+\alpha_l)<4$ and $4(1+\alpha_l)>3$, which guarantees that the leading order term in \eqnref{asy expansion of H_s Size} can be distinguished from the higher order term for all $l=1, 2, \ldots, l_0$. As an illustrative example, we suppose there are two anomalies, say $l_1$ and $l_2$, with one is growing ($l_1<0$) and the other is decaying ($l_2>0$). If there holds that $4(1+\alpha_{l_1})=3(1+\alpha_{l_2})$, or $3\alpha_{l_2}-4\alpha_{l_1}=1$, then the higher order term of the anomaly $l_1$ is the same as the leading order term of the anomaly $l_2$, which means that $l_2$ can not be distinguished, and thus the shrinkage of the anomaly $l_2$ can not be recovered by only using the measurement \eqnref{inverseproblemTime}.
\end{rem}

\section {Further asymptotic analysis}
In this section, we make further analyse on the leading-order term of the asymptotic expansion in \eqref{asy expansion of H_s Size}. As remarked in Remark~\ref{rem:ll1}, we show that the auxiliary terms $\mathbf{P}_l\mathbf{H}_0^{(0)}(\mathbf{z}_l)$ can be well approximated by the measurement data in \eqref{inverse problem frequency}. This shall be of critical importance if one intends to develop a practical reconstruction scheme by using our unique recovery results in the next section. We shall mainly make use of the orthogonality of spherical harmonics functions. Before that, we first present some axillary results.

\begin{lem}\label{GradientGamma}{(Lemma 3.9 in \cite{deng2018identifying})}
	Let $\mathbf{z}\in\mathbb{R}^3$ be fixed. Let $\mathbf{x}\in \partial B_R$ and suppose $\|\mathbf{z}\|<R$. There holds the following asymptotic expansion
	\begin{equation}\label{eq:3.46}
		\nabla\Gamma_0(\mathbf{x}-\mathbf{z})=\sum_{n=0}^{\infty}\sum_{m=-n}^{n}\frac{(n+1)Y^m_n(\hat{\mathbf{x}})\hat{\mathbf{x}}-\nabla_sY^m_n(\hat{\mathbf{x}})}{(2n+1)R^{n+2}}\overline{Y_n^m(\hat{\mathbf{z}})}\|{\mathbf{z}}\|^n,
	\end{equation}
	where $\hat{\mathbf{z}}=\mathbf{z}/\|\mathbf{z}\|$ and $\hat{\mathbf{x}}=\mathbf{x}/\|\mathbf{x}\|$. $Y^m_n$ is the spherical harmonics of order $m$ and degree $n$.
\end{lem}

\begin{lem}\label{GradientGradientGamma}{(Equation (3.9) in \cite{cdengMHD})}
	Let $\mathbf{z}_l\in\mathbb{R}^3$ be fixed. Let $\mathbf{x}\in \partial B_R$ and suppose $\|\mathbf{z}\|<R$. There holds the following asymptotic expansion
	\begin{equation}\label{eq:GradientGradientGamma}
		\nabla^2\Gamma_0(\mathbf{x}-\mathbf{z})=\sum_{n=0}^{\infty}\sum_{m=-n}^{n}\frac{\mathbf{A}^m_n(\hat{\mathbf{x}})}{(2n+1)R^{n+2}}\overline{Y_n^m(\hat{\mathbf{z}})}\|{\mathbf{z}}\|^n,
	\end{equation}
	where \begin{equation}\label{A}
		\mathbf{A}^m_n:=(n+1)(\hat{\mathbf{x}}\nabla_sY^m_n(\hat{\mathbf{x}})^T+Y^m_n(\hat{\mathbf{x}})(I-\hat{\mathbf{x}}\hat{\mathbf{x}}^T))-\nabla^2_sY^m_n(\hat{\mathbf{x}})-(n+2)\mathbf{N}^m_{n+1}(\hat{\mathbf{x}})\hat{\mathbf{x}}^T,
	\end{equation}
	 and  $\hat{\mathbf{z}}=\mathbf{z}/\|\mathbf{z}\|$ and $\hat{\mathbf{x}}=\mathbf{x}/\|\mathbf{x}\|$. $Y^m_n$ is the spherical harmonics of order $m$ and degree $n$.
\end{lem}
Using the above results, we can further estimate the leading order term in \eqnref{asy expansion of H_s Size}.
By substituting \eqref{eq:GradientGradientGamma} into \eqref{asy expansion of H_s Size}, one has
\begin{equation}\label{eq:asy of H final spherical Harmonics A}
\begin{aligned}
&(\mathbf{H}_{s}^{(0)}-\mathbf{H}^{(0)})(\mathbf{x})
\\&=	
\delta^3\sum_{l=1}^{l_0}(\delta^{3\alpha_l}-1)\left(\sum_{n=0}^{\infty}\sum_{m=-n}^{n}\frac{\mathbf{A}^m_n(\hat{\mathbf{x}})}{(2n+1)\|\mathbf{x}\|^{n+2}}
\overline{Y_n^m(\hat{\mathbf{z}}_l)}\|{\mathbf{z}}_l\|^n\mathbf{P}_l\mathbf{H}^{(0)}_0(\mathbf{z}_l)\right)+o(\delta^3).
\end{aligned}
\end{equation}
In the sequel, we define
\begin{equation}\label{eq:4.5}
	\mathbf{N}^m_{n+1}(\hat{\mathbf{x}}):=(n+1)Y^m_n(\hat{\mathbf{x}})\hat{\mathbf{x}}-\nabla_sY^m_n(\hat{\mathbf{x}}),
\end{equation}
for $n\in\mathbb{N}\cup{0}$ and $m=-n,-n+1,\dots,n-1,n,$ and
\begin{equation}\label{Q,T}
	\mathbf{Q}^m_{n-1}(\hat{\mathbf{x}}):=\nabla_sY_n^m(\hat{\mathbf{x}})+nY_n^m(\hat{\mathbf{x}})\hat{\mathbf{x}},\quad\mathbf{T}^m_n(\hat{\mathbf{x}}):=\nabla_sY^m_n(\hat{\mathbf{x}})\times\hat{\mathbf{x}},
\end{equation}
for $n\in\mathbb{N}$ and $m=-n,-n+1,\dots,n-1,n.$ Note that $	\mathbf{N}^m_{n+1}(\hat{\mathbf{x}})$, $\mathbf{Q}^m_{n-1}(\hat{\mathbf{x}})$ and $\mathbf{T}^m_n(\hat{\mathbf{x}})$ are spherical harmonics of order $n$. We can further establish the following lemmas.

\begin{lem}\label{lem:system for N}
	Let $\mathbf{P}_l$ be defined in \eqref{polarization}. Then there holds
	\begin{equation}\label{eq:system for N}
	\begin{aligned}
	&\int_{\mathbb{S}^2}\overline{\mathbf{N}_{2}(\hat{\mathbf{x}})}\cdot(\mathbf{H}_{s}^{(0)}-\mathbf{H}^{(0)})(\mathbf{x})ds\\
	=	&	
	\frac{\delta^3}{\|\mathbf{x}\|^3}\left({\mathbf{C}_0}\sum_{l=1}^{l_0}(\delta^{3\alpha_l}-1)\mathbf{P}_l\mathbf{H}_0^{(0)}(\mathbf{z}_l)
	+\mathcal{O}(\|\mathbf{x}\|^{-2})\right)+o(\delta^3),
	\end{aligned}
	\end{equation}
	where $\mathbf{C}_0$ and $\mathbf{N}_2(\hat{\mathbf{x}})$ are 3{-by-}3 matrices given by
	\begin{equation}
	\mathbf{C}_0:=\left(-4\mathbf{a}_{1,0}^{-1,0}+\overline{\mathbf{a}_{0,1}^{0,-1}},-4\mathbf{a}_{1,0}^{0,0}+\overline{\mathbf{a}_{0,1}^{0,0}},-4\mathbf{a}_{1,0}^{1,0}+\overline{\mathbf{a}_{0,1}^{0,1}}\right)^T,
	\end{equation}
	and
	\begin{equation}
	\mathbf{N}_2(\hat{\mathbf{x}}):=\left(\mathbf{N}_2^{-1}(\hat{\mathbf{x}}),\mathbf{N}_2^0(\hat{\mathbf{x}}),\mathbf{N}_2^{1}(\hat{\mathbf{x}})\right)
	\end{equation}
	respectively. $\mathbf{a}^{m',m}_{n',n}$ is defined in \eqref{eq:a}.
\end{lem}

\begin{proof}
	From \eqref{A}, one has
	\begin{equation}\label{Axi}
	\begin{aligned}
	\mathbf{A}^m_n(\hat{\mathbf{x}})\xi=&(n+1)\hat{\mathbf{x}}\nabla_sY^m_n(\hat{\mathbf{x}})^T\xi+(n+1)Y^m_n(\hat{\mathbf{x}})\xi-(n+1)(n+3)Y^m_n(\hat{\mathbf{x}})\hat{\mathbf{x}}\hat{\mathbf{x}}^T\xi\\&-\nabla_s(\nabla_sY^m_n(\hat{\mathbf{x}})^T\xi)+(n+2)\nabla_sY^m_n(\hat{\mathbf{x}})\hat{\mathbf{x}}^T\xi
	\end{aligned}
	\end{equation}
	where $\xi\in \mathbb{R}^3$. Note that $\mathbf{A}^m_n(\hat{\mathbf{x}})\xi$ is a linear combination of the spherical harmonics $\{\mathbf{N}_{n+1}^m(\hat{\mathbf{x}})\}$ and $\{\mathbf{Q}_{n-1}^{m}(\hat{\mathbf{x}})\}$. We show that $\mathbf{N}_2^{m'}(\hat{\mathbf{x}}),m'=-1,0,1$ are orthogonal to $\mathbf{A}_n^m(\hat{\mathbf{x}})\xi$ for $\xi\in\mathbb{R}^3$ and $n\neq0,2$, where $\mathbf{A}_n^m(\hat{\mathbf{x}})$ is defined in \eqref{A}.
	By vector calculus identity and straightforward computations, one can obtain that
	\begin{equation}\label{C.10}
	\begin{aligned}
	&\int_{\mathbb{S}^2}\overline{\mathbf{N}^{m'}_{n'+1}(\hat{\mathbf{x}})}\cdot(\mathbf{A}^m_n(\hat{\mathbf{x}})\xi)ds\\
	=&\left((n'+1)(n'+n+1)\mathbf{a}^{m',m}_{n',n}-(n'+1)(n'+n+1)(n+2)\mathbf{b}^{m',m}_{n',n}+\overline{\mathbf{a}^{m,m'}_{n,n'}}\right)^T\xi,
	\end{aligned}	
	\end{equation}
	where
	\begin{equation}\label{eq:a}
	\mathbf{a}^{m',m}_{n',n}:=\int_{\mathbb{S}^2}\overline{Y_{n'}^{m'}(\hat{\mathbf{x}}})\nabla_sY_n^m(\hat{\mathbf{x}})ds,
	\end{equation}
	and
	\begin{equation}\label{eq:b}
	\mathbf{b}^{m',m}_{n',n}:=\int_{\mathbb{S}^2}\overline{Y_{n'}^{m'}(\hat{\mathbf{x}})}Y_{n}^{m}(\hat{\mathbf{x}})\hat{\mathbf{x}}ds.
	\end{equation}
	By using the following elementary result (cf. \cite{nedelec2001acoustic})
	\begin{equation}
	\mathbf{a}_{n',n}^{m',m}=\mathbf{b}_{n',n}^{m',m}=0,\quad\text{for any }n'\neq n-1,n+1\ \text{and }m'\neq m-1,m+1,
	\end{equation}
	one finally obtains
	\begin{equation}\label{AxiCD}
	\begin{split}
	\mathbf{A}_n^m(\hat{\mathbf{x}})\xi=&\sum_{m'=m-1}^{m+1}\Big((\mathbf{c}_{n-1,n}^{m',m})^T\xi\mathbf{N}_n^{m'}(\hat{\mathbf{x}})+(\mathbf{c}_{n+1,n}^{m',m})^T\xi\mathbf{N}_{n+2}^{m'}(\hat{\mathbf{x}})\\&+(\mathbf{d}_{n-1,n}^{m',m})^T\xi\mathbf{Q}_0^{m'}(\hat{\mathbf{x}})+(\mathbf{d}_{n+1,n}^{m',m})^T\xi\mathbf{Q}_n^{m'}(\hat{\mathbf{x}})\Big),
	\end{split}
	\end{equation}
	where
	\begin{equation}\label{c}
	\mathbf{c}_{n',n}^{m',m}:=\frac{(n'+1)(n'+n+1)\mathbf{a}_{n',n}^{m',m}-(n'+1)(n'+n+1)(n+2)\mathbf{b}_{n',n}^{m',m}+\overline{\mathbf{a}_{n,n'}^{m,m'}}}{(n'+1)(2n'+1)},
	\end{equation}
	and
	\begin{equation}\label{d}
	\mathbf{d}_{n',n}^{m',m}:=\frac{n'(n-n')\mathbf{a}_{n',n}^{m',m}-n'(n-n')(n+2)\mathbf{b}_{n',n}^{m',m}-\overline{\mathbf{a}_{n,n'}^{m,m'}}}{n'(2n'+1)}.
	\end{equation}
	By using \eqref{Axi} the
	orthogonality of the vectorial spherical harmonics, one has
	\begin{equation}\label{orthog}
	\int_{\mathbb{S}^2}\overline{\mathbf{N}_2^{m'}(\hat{\mathbf{x}})}\cdot(\mathbf{A}_n^m(\hat{\mathbf{x}})\xi)ds=0\quad n\neq 0,2.
	\end{equation}
	Note that if $n=2$, then by $\eqref{Axi}$, one has
	\begin{equation}
	\begin{split}
	\mathbf{A}_2^m(\hat{\mathbf{x}})\xi=&\sum_{m'=m-1}^{m+1}\Big((\mathbf{c}_{1,2}^{m',m})^T\xi\mathbf{N}_2^{m'}(\hat{\mathbf{x}})+(\mathbf{c}_{3,2}^{m',m})^T\xi\mathbf{N}_4^{m'}(\hat{\mathbf{x}})\\&+(\mathbf{d}_{1,2}^{m',m})^T\xi\mathbf{Q}_0^{m'}(\hat{\mathbf{x}})+(\mathbf{d}_{3,2}^{m',m})^T\xi\mathbf{Q}_2^{m'}(\hat{\mathbf{x}})\Big),
	\end{split}
	\end{equation}
	where $\mathbf{c}_{n',n}^{m',m}$ and $\mathbf{d}_{n',n}^{m',m}$	are defined in \eqref{c} and \eqref{d}, respectively. Thus,
	\begin{equation}\label{3.20}
	\int_{\mathbb{S}^2}\overline{\mathbf{N}_2^{m'}(\hat{\mathbf{x}})}\cdot(\mathbf{A}_2^m(\hat{\mathbf{x}})\xi)ds=6(\mathbf{c}_{1,2}^{m',m})^T\xi.
	\end{equation}
	By using the orthogonality of $\mathbf{N}_3^m(\hat{\mathbf{x}})$ and $Y_1^{m'},m,m'\in \{-1,0,1\}$, one has
	\begin{equation}\label{orthogothal}
	\int_{\mathbb{S}^2}\overline{Y_1^{m'}(\hat{\mathbf{x}})}(3Y_2^m(\hat{\mathbf{x}})\hat{\mathbf{x}}-\nabla_sY_2^m(\hat{\mathbf{x}}))ds=0,
	\end{equation}
	which indicates $\mathbf{a}_{1,2}^{m',m}=3\mathbf{b}_{1,2}^{m',m}$, where $\mathbf{a}_{n',n}^{m',m}$ and $\mathbf{b}_{n',n}^{m',m}$ are defined in \eqref{eq:a} and \eqref{eq:b}, respectively. Together with \eqref{3.20} and \eqref{c}, one has
	\begin{equation}\label{A_2}
	\int_{\mathbb{S}^2}\overline{\mathbf{N}_2^{m'}(\hat{\mathbf{x}})}\cdot(\mathbf{A}_2^m(\hat{\mathbf{x}})\xi)ds=\frac{1}{3}(-8\mathbf{a}_{1,2}^{m',m}+3\overline{\mathbf{a}_{2,1}^{m,m'}})^T\xi.
	\end{equation}
	Similarly, for $n=0$, one obtains that
	\begin{equation}\label{A_0}
	\begin{split}
	\mathbf{A}_0^0(\hat{\mathbf{x}})\xi=&\sum_{m'=-1}^{1}((\mathbf{c}_{-1,0}^{m',0})^T\xi\mathbf{N}_0^{m'}(\hat{\mathbf{x}})+(\mathbf{c}_{1,0}^{m',0})^T\xi\mathbf{N}_2^{m'}(\hat{\mathbf{x}})\\&+(\mathbf{d}_{-1,0}^{m',0})^T\xi\mathbf{Q}_{-2}^{m'}(\hat{\mathbf{x}})+(\mathbf{d}_{1,0}^{m',0})^T\xi\mathbf{Q}_0^{m'}(\hat{\mathbf{x}})).
	\end{split}
	\end{equation}
	Thus,
	\begin{equation}\label{N_2dotA_4}
	\int_{\mathbb{S}^2}\overline{\mathbf{N}_2^{m'}(\hat{\mathbf{x}})}\cdot(\mathbf{A}_0^0(\hat{\mathbf{x}})\xi)ds=6(\mathbf{c}_{1,0}^{m',0})^T\xi.
	\end{equation}
	By using the orthogonality of $\mathbf{N}_1^m(\hat{\mathbf{x}})$ and $Y_1^{m'},m,m'\in \{-1,0,1\}$, one has
	\begin{equation}\label{orthogonal 10}
	\int_{\mathbb{S}^2}\overline{Y_1^{m'}(\hat{\mathbf{x}})}(Y_0^0(\hat{\mathbf{x}})\hat{\mathbf{x}}-\nabla_sY_0^0(\hat{\mathbf{x}}))ds=0,
	\end{equation}
	which indicates $\mathbf{a}_{1,0}^{m',0}=\mathbf{b}_{1,0}^{m',0}$. Combining \eqref{N_2dotA_4} and \eqref{c}, one has
	\begin{equation}\label{}
	\int_{\mathbb{S}^2}\overline{\mathbf{N}_2^{m'}(\hat{\mathbf{x}})}\cdot(\mathbf{A}_0^0(\hat{\mathbf{x}})\xi)ds=(4\mathbf{a}_{1,0}^{m',0}-8\mathbf{b}_{1,0}^{m',0}+\overline{\mathbf{a}_{0,1}^{0,m'}})^T\xi=(-4\mathbf{a}_{1,0}^{m',0}+\overline{\mathbf{a}_{0,1}^{0,m'}})^T\xi.
	\end{equation}
	Finally, by taking the inner product of \eqref{eq:asy of H final spherical Harmonics A} and $\mathbf{N}_2^{m'}(\hat{\mathbf{x}})$ in $L^2(\mathbb{S})$ and by using \eqref{A_2} and \eqref{A_0}, one has
	\begin{equation}\label{rearrange}
	\begin{aligned}
	&\int_{\mathbb{S}^2}\overline{\mathbf{N}_2(\hat{\mathbf{x}})}\cdot(\mathbf{H}_{s}^{(0)}-\mathbf{H}^{(0)})(\mathbf{x})ds\\
	&\quad=\frac{\delta^3}{\|\mathbf{x}\|^3}(-4\mathbf{a}_{1,0}^{m',0}+\overline{\mathbf{a}_{0,1}^{0,m'}})^T\sum_{l=1}^{l_0}(\delta^{3\alpha_l}-1)\mathbf{P}_l\mathbf{H}_0^{(0)}(\mathbf{z}_l)\\
	&\quad+\frac{1}{15}\frac{\delta^3}{\|\mathbf{x}\|^5}\sum_{m=-2}^{2}(-8\mathbf{a}_{1,2}^{m',m}+
	3\overline{\mathbf{a}_{2,1}^{m,m'}})^T\sum_{l=1}^{l_0}(\delta^{3\alpha_l}-1)\mathbf{P}_l\mathbf{H}_0^{(0)}(\mathbf{z}_l)+o(\delta^3),
	\end{aligned}
	\end{equation}
	By rearranging the terms in equation \eqref{rearrange}, we finally arrive at \eqref{eq:system for N}, which completes the proof.
\end{proof}

\begin{lem}\label{lem:system for Q}
Let $P_l$ be defined in \eqref{polarization}. Then there holds
\begin{equation}\label{eq:system for Q}
\begin{aligned}
&\int_{\mathbb{S}^2}\overline{\mathbf{Q}_{0}(\hat{\mathbf{x}})}\cdot(\mathbf{H}_{s}^{(0)}-\mathbf{H}^{(0)})(\mathbf{x})ds\\
=&
\frac{\delta^3}{\|\mathbf{x}\|^3}\left({\mathbf{D}_0}\sum_{l=1}^{l_0}(\delta^{3\alpha_l}-1)\mathbf{P}_l\mathbf{H}_0^{(0)}(\mathbf{z}_l)
+\mathcal{O}(\|\mathbf{x}\|^{-2})\right)+o(\delta^3),
\end{aligned}
\end{equation}
where $\mathbf{D}_0$ and $\mathbf{Q}_0(\hat{\mathbf{x}})$ are 3{-by-}3 matrices given by
\begin{equation}
\mathbf{D}_0:=(\mathbf{a}_{1,0}^{-1,0}-\overline{\mathbf{a}_{0,1}^{0,-1}},\mathbf{a}_{1,0}^{0,0}-\overline{\mathbf{a}_{0,1}^{0,0}},\mathbf{a}_{1,0}^{-1,0}-\overline{\mathbf{a}_{0,1}^{0,1}})^T,
\end{equation}
and
\begin{equation}
\mathbf{Q}_0(\hat{\mathbf{x}}):=\left(\mathbf{Q}_0^{-1}(\hat{\mathbf{x}}),\mathbf{Q}_0^0(\hat{\mathbf{x}}),\mathbf{Q}_0^{1}(\hat{\mathbf{x}})\right)
\end{equation}
respectively. $\mathbf{a}^{m',m}_{n',n}$ is defined in \eqref{eq:a}.
\end{lem}
\begin{proof}
The proof of \eqref{eq:system for Q} follows from a similar argument to that in the proof of Lemma \ref{lem:system for N}. We show that $\mathbf{Q}_0^{m'}(\hat{\mathbf{x}}),m'=-1,0,1$ are orthogonal to $\mathbf{A}_n^m(\hat{\mathbf{x}})\xi$ for $\xi\in\mathbb{R}^3$ and $n\neq0,2$, where $\mathbf{A}_n^m(\hat{\mathbf{x}})$ is defined in \eqref{A}.
By vector calculus identity and straightforward computations, one can obtain that
\begin{equation}\label{C.11}
\begin{aligned}
&\int_{\mathbb{S}^2}\overline{\mathbf{Q}^{m'}_{n'-1}(\hat{\mathbf{x}})}\cdot(\mathbf{A}^m_n(\hat{\mathbf{x}})\xi)ds\\
=&\left(n'(n-n')\mathbf{a}^{m',m}_{n',n}-n'(n-n')(n+2)\mathbf{b}^{m',m}_{n',n}-\overline{\mathbf{a}^{m,m'}_{n,n'}}\right)^T\xi,
\end{aligned}	
\end{equation}
where $\mathbf{a}^{m',m}_{n',n}$ and $\mathbf{b}^{m',m}_{n',n}$ are defined in \eqref{eq:a} and \eqref{eq:b} respectively.
By using \eqref{AxiCD} the
orthogonality of the vectorial spherical harmonics, one has
\begin{equation}\label{orthogQ}
\int_{\mathbb{S}^2}\overline{\mathbf{Q}_0^{m'}(\hat{\mathbf{x}})}\cdot(\mathbf{A}_n^m(\hat{\mathbf{x}})\xi)ds=0\quad n\neq 0,2.
\end{equation}
Note that if $n=2$, then by $\eqref{AxiCD}$, one has
\begin{equation}\label{3.20Q}
\int_{\mathbb{S}^2}\overline{\mathbf{Q}_0^{m'}(\hat{\mathbf{x}})}\cdot(\mathbf{A}_2^m(\hat{\mathbf{x}})\xi)ds=3(\mathbf{d}_{1,2}^{m',m})^T\xi,
\end{equation}
where $\mathbf{d}_{n',n,}^{m',m}$ is defined in \eqref{d}. Together with \eqref{orthogothal}, one has
\begin{equation}\label{A_2Q}
\begin{aligned}
\int_{\mathbb{S}^2}\overline{\mathbf{Q}_0^{m'}(\hat{\mathbf{x}})}\cdot(\mathbf{A}_2^m(\hat{\mathbf{x}})\xi)ds=&3(\mathbf{d}_{1,2}^{m',m})^T\xi\\
&=3\left(\frac{(1)(2-1)\mathbf{a}_{1,2}^{m',m}-(1)(2-1)(4)\mathbf{b}_{1,2}^{m',m}-\overline{\mathbf{a}_{2,1}^{m,m'}}}{(1)(2+1)}\right)^T\xi\\
&=(\mathbf{a}_{1,2}^{m',m}-4\mathbf{b}_{1,2}^{m',m}-\overline{\mathbf{a}_{2,1}^{m,m'}})^T\xi\\
&=-\frac{1}{3}(\mathbf{a}_{1,2}^{m',m}+3\overline{\mathbf{a}_{2,1}^{m,m'}})^T\xi.
\end{aligned}
\end{equation}
Similarly, for $n=0$, one obtains that
\begin{equation}\label{Q_0dotA_2}
\int_{\mathbb{S}^2}\overline{\mathbf{Q}_0^{m'}(\hat{\mathbf{x}})}\cdot(\mathbf{A}_0^0(\hat{\mathbf{x}})\xi)ds=3(\mathbf{d}_{1,0}^{m',0})^T\xi.
\end{equation}
Using \eqref{orthogonal 10} and \eqref{d}, one has
\begin{equation}\label{}
\int_{\mathbb{S}^2}\overline{\mathbf{Q}_0^{m'}(\hat{\mathbf{x}})}\cdot(\mathbf{A}_0^0(\hat{\mathbf{x}})\xi)ds=(-\mathbf{a}_{1,0}^{m',0}+2\mathbf{b}_{1,0}^{m',0}-\overline{\mathbf{a}_{0,1}^{0,m'}})^T\xi=(\mathbf{a}_{1,0}^{m',0}-\overline{\mathbf{a}_{0,1}^{0,m'}})^T\xi.
\end{equation}
Finally, by taking the inner product of \eqref{eq:asy of H final spherical Harmonics A} and $\mathbf{Q}_0^{m'}(\hat{\mathbf{x}})$ in $L^2(\mathbb{S})$ and by using \eqref{A_2} and \eqref{A_0}, one has
\begin{equation}\label{rearrangeQ}
\begin{aligned}
&\int_{\mathbb{S}^2}\overline{\mathbf{Q}_0(\hat{\mathbf{x}})}\cdot(\mathbf{H}_{s}^{(0)}-\mathbf{H}^{(0)})(\mathbf{x})ds\\
&\quad=\frac{\delta^3}{\|\mathbf{x}\|^3}(\mathbf{a}_{1,0}^{m',0}-\overline{\mathbf{a}_{0,1}^{0,m'}})^T\sum_{l=1}^{l_0}(\delta^{3\alpha_l}-1)\mathbf{P}_l\mathbf{H}_0^{(0)}(\mathbf{z}_l)\\
&\quad-\frac{1}{15}\frac{\delta^3}{\|\mathbf{x}\|^5}\sum_{m=-2}^{2}(\mathbf{a}_{1,2}^{m',m}
+3\overline{\mathbf{a}_{2,1}^{m,m'}})^T\sum_{l=1}^{l_0}(\delta^{3\alpha_l}-1)\mathbf{P}_l\mathbf{H}_0^{(0)}(\mathbf{z}_l)+o(\delta^3).
\end{aligned}
\end{equation}
By rearranging the terms in equation \eqref{rearrange}, we finally arrive at \eqref{eq:system for N}. The proof is complete.
\end{proof}
The following expressions are the general forms of Lemmas \ref{lem:system for N} and \ref{lem:system for Q} respectively.
By \eqref{C.10} and taking the inner product of \eqref{eq:asy of H final spherical Harmonics A} and \eqref{eq:4.5} in $L^2(\mathbb{S})$, one can obtain that
\begin{equation}\label{accurate form of N dot A}
\begin{aligned}
&\int_{\mathbb{S}^2}\overline{\mathbf{N}_{n''+1}^{m''}(\hat{\mathbf{x}})}\cdot(\mathbf{H}_{s}^{(0)}-\mathbf{H}^{(0)})(\mathbf{x})ds\\
=	&	\delta^3\sum_{l=1}^{l_0}(\delta^{3\alpha_l}-1)\left(\sum_{n=0}^{\infty}\sum_{m=-n}^{n}\frac{	\int_\mathbb{S}\overline{\mathbf{N}_{n''+1}^{m''}(\hat{\mathbf{x}})}\cdot(\mathbf{A}^m_n(\hat{\mathbf{x}})\xi)ds}{(2n+1)\|\mathbf{x}\|^{n+2}}
\overline{Y_n^m(\hat{\mathbf{z}}_l)}\|{\mathbf{z}}_l\|^n\mathbf{P}_l\mathbf{H}^{(0)}_0(\mathbf{z}_l)\right)\\&+o(\delta^3),
\end{aligned}
\end{equation}
where $\mathbb{S}$ stands for the unit sphere. Similarly, one can also obtain
that
\begin{equation}\label{accurate form of Q dot A}
\begin{aligned}
&\int_{\mathbb{S}^2}\overline{\mathbf{Q}_{n''-1}^{m''}(\hat{\mathbf{x}})}\cdot(\mathbf{H}_{s}^{(0)}-\mathbf{H}^{(0)})(\mathbf{x})ds\\
=	&	\delta^3\sum_{l=1}^{l_0}(\delta^{3\alpha_l}-1)\left(\sum_{n=0}^{\infty}\sum_{m=-n}^{n}\frac{	\int_{\mathbb{S}^2}\overline{\mathbf{Q}_{n''-1}^{m''}(\hat{\mathbf{x}})}\cdot(\mathbf{A}^m_n(\hat{\mathbf{x}})\xi)ds}{(2n+1)\|\mathbf{x}\|^{n+2}}
\overline{Y_n^m(\hat{\mathbf{z}}_l)}\|{\mathbf{z}}_l\|^n\mathbf{P}_l\mathbf{H}^{(0)}_0(\mathbf{z}_l)\right)\\&+o(\delta^3).
\end{aligned}
\end{equation}

We remark that the coefficient of higher order $\mathcal{O}(\|\mathbf{x}\|^{-2})$ in \eqref{eq:system for N} is accurately given in \eqref{rearrange}, and the general form is given in \eqref{accurate form of N dot A}. Similar argument can be applied to the spherical harmonics $\mathbf{Q}_{n-1}^{m}$. If the quantity $\|\mathbf{z}_l\|/\|\mathbf{x}\|$ in \eqref{rearrange} is small enough, that is, the radius of the measurement surface is much bigger than the length of the position vector of the magnetized anomaly, then the unknown term $\mathbf{P}_l\mathbf{H}_0^{(0)}(\mathbf{z}_l)$ is easy to recover. In the following, we consider a particular case with $\Gamma=\partial B_R$, with $B_R$ a sufficiently large central ball containing $\Sigma$.

\begin{lem}
	Suppose that $\Gamma=\partial B_R$ and let $\mathbf{P}_l$ be defined in \eqref{polarization}. Then there hold the following relationships,
	\begin{equation}\label{eq:lem3.5n}
		\sum_{l=1}^{l_0}(\delta^{3\alpha_l}-1)\mathbf{P}_l\mathbf{H}_0^{(0)}(\mathbf{z}_l)\approx\delta^{-3}R^3{\mathbf{C}_0}^{-1}\int_{\mathbb{S}^2}\overline{\mathbf{N}_2(\hat{\mathbf{x}})}(\mathbf{H}_s^{(0)}-\mathbf{H}^{(0)})(\mathbf{x})ds,
	\end{equation}
	and
	\begin{equation}\label{eq:lem3.5q}
	\sum_{l=1}^{l_0}(\delta^{3\alpha_l}-1)\mathbf{P}_l\mathbf{H}_0^{(0)}(\mathbf{z}_l)\approx\delta^{-3}R^3{\mathbf{D}_0}^{-1}\int_{\mathbb{S}^2}\overline{\mathbf{Q}_0(\hat{\mathbf{x}})}(\mathbf{H}_s^{(0)}-\mathbf{H}^{(0)})(\mathbf{x})ds.
	\end{equation}
\end{lem}
 \begin{proof}
 	The proof of \eqref{eq:lem3.5n} and \eqref{eq:lem3.5q} follows from a similar argument to that in the proof of Lemma 3.3 in \cite{cdengMHD}. By straightforward calculations, one can show that $\mathbf{C}_0$ and $\mathbf{D}_0$ are invertible. The proof is completed by directly using \eqref{eq:system for N}.
 \end{proof}


\section{Unique recovery results for magnetic anomaly detections}
In this section, we present the main unique recovery results in identifying the magnetic anomalies for the inverse problem \eqref{inverseproblem FinalFequency}, which are contained in Theorems \ref{thm:singleuniqueness} and \ref{thm:multiuniqueness}. In what follows, we let $D_l^{(1)}$ and $D_l^{(2)}, l=1,2,\dots,l_0,$ be two sets of magnetic anomalies, which satisfy \eqref{magnetized anomalies} with $\mathbf{z}_l$ replaced by $\mathbf{z}_l^{(1)}$ and $\mathbf{z}_l^{(2)}$, respectively, and $s_l=\delta^\alpha_l$ replaced by $s_l^{(1)}=\delta^{\alpha^{(1)}_l}$ and $s_l^{(2)}=\delta^{\alpha^{(2)}_l}$, respectively. The material parameters $\epsilon_l,\sigma_l,\gamma_l$ and $\mu_l$ are also replaced by $\epsilon_l^{(1)},\sigma_l^{(1)},\gamma_l^{(1)},\mu_l^{(1)}$ and $\epsilon_l^{(2)},\sigma_l^{(2)},\gamma_l^{(2)},\mu_l^{(2)}$, for $D_l^{(1)}$ and $D_l^{(2)},l=1,2,\dots,l_0$ respectively. Let $\mathbf{H}_{s,j},j=1,2$ be the solutions to \eqref{eq:MaterialDisturbution} and \eqref{frequencymaxwell} with $D_l$ replaced by $D_l^{(1)}$ and $D_l^{(2)}$, respectively. Let $\mathbf{H}_{j},j=1,2$ be the solutions to \eqref{materialdistfree} and \eqref{frequencymaxwell} with $D_l$ replaced by $D_l^{(1)}$ and $D_l^{(2)}$, respectively. Let $\mathbf{H}_{s,j}^{(0)}$ be the leading-order terms of $\mathbf{H}_{s,j}$ with respect to $\omega\ll1$. Let $\mathbf{H}_{s}^{(0)}$ be the leading-order terms of $\mathbf{H}_{s}$ with respect to $\omega\ll1$.  Denote by $\mathbf{D}_l^{(1)}, \mathbf{D}_l^{(2)}, \mathbf{M}_l^{(1)}, \mathbf{M}_l^{(2)}, \mathbf{P}_l^{(1)}$ and $\mathbf{P}_l^{(2)}$ the polarization tensors for $D_l^{(1)}$ and $D_l^{(2)}$, respectively, $l=1,2,\dots,l_0$.

Before that, we first present an important axillary result.
\begin{lem}\label{lem:sumND}
   Suppose $D_l, l=1,2,\dots,l_0$ are defined in \eqref{magnetized anomalies} with $\delta\in\mathbb{R}_+$ sufficiently small and $s_l=\delta^{\alpha_l},l=1,2,\dots,l_0$ with $-1/4<\alpha_l< 1/3$. Suppose also that $3(\alpha_{l_1}+1)\neq 4(\alpha_{l_2}+1)$ for any $l_1, l_2\in \{1, 2, \ldots, l_0\}$. Let $(\mathbf{E}_s,\mathbf{H}_s)$ be the solution to \eqref{eq:MaterialDisturbution} and \eqref{frequencymaxwell}. If there hold
	\begin{equation}\label{eq:4.2}
		\nu\cdot\mathbf{H}_{s,1}=\nu\cdot\mathbf{H}_{s,2}\neq0\quad\text{on }\Gamma,
	\end{equation}
	and
	\begin{equation}\label{eq:4.2b}
\nu\cdot\mathbf{H}_{1}=\nu\cdot\mathbf{H}_{2}\neq0\quad\text{on }\Gamma,
\end{equation}	
	then one has
	\begin{equation}\label{eq:4.3}
		\sum_{m=-1}^{n}\mathbf{N}^m_{n+1}(\hat{\mathbf{x}})^T\mathbf{d}_1^{n,m}=\sum_{m=-1}^{n}\mathbf{N}^m_{n+1}(\hat{\mathbf{x}})^T\mathbf{d}_2^{n,m},\quad\hat{\mathbf{x}}\in\mathbb{S}^2,
	\end{equation}
	for any $n\in\mathbb{N}\cup\{0\}$, where
	\begin{equation}\label{eq:4.4}
		\mathbf{d}^{n,m}_j=\sum_{l=1}^{l_0}s_l^{(j)}\overline{Y^m_n(\hat{\mathbf{z}}_l^{(j)})}\|\mathbf{z}_l^{(j)}\|^n(\mathbf{P}^{(j)}_l)\mathbf{H}_0^{(0)}(\mathbf{z}_l^{(j)}),\quad j=1,2.
	\end{equation}
Here, $\hat{\mathbf{z}}_l^{(j)}=\mathbf{z}_l^{(j)}/\|\mathbf{z}_l^{(j)}\|,l=1,2\dots,l_0,j=1,2$ , $\hat{\mathbf{x}}=\mathbf{x}/\|\mathbf{x}\|$ and $Y_n^m$ is the spherical harmonics of order $m$ and degree $n$.
\end{lem}

\begin{proof}
	First, by using \eqref{eq:4.2} and \eqref{eq:4.2b} and unique continuation, one has
	\begin{equation}
		\mathbf{H}_{s,1}=\mathbf{H}_{s,2}\quad\text{in }\mathbb{R}^3\setminus\overline{\Sigma},
	\end{equation}
	and
	\begin{equation}
	\mathbf{H}_{1}=\mathbf{H}_{2}\quad\text{in }\mathbb{R}^3\setminus\overline{\Sigma}.
	\end{equation}
	Then from \eqref{eq:asy of H_sj in term of polar}, one has
	\begin{equation}\label{eq:4.6}
	\begin{aligned}
	   &\sum_{l=1}^{l_0}s_l^{(1)}\left(\nabla^2\Gamma_0(\mathbf{x}-\mathbf{z}_l^{(1)})\mathbf{P}_l^{(1)}\mathbf{H}^{(0)}_0(\mathbf{z}_l^{(1)})\right)\\
	   &\quad=\sum_{l=1}^{l_0}s_l^{(2)}\left(\nabla^2\Gamma_0(\mathbf{x}-\mathbf{z}_l^{(2)})\mathbf{P}_l^{(2)}\mathbf{H}^{(0)}_0(\mathbf{z}_l^{(2)})\right)\quad\text{in }\mathbb{R}^3\setminus\overline{\Sigma}.
	\end{aligned}
	\end{equation}
	Suppose $R\in \mathbb{R}_+$ is sufficiently large such that $\Sigma\Subset B_{R}$. By \eqref{eq:4.6} one readily has
	\begin{equation}\label{eq:4.7}
	\begin{aligned}
	    &\nu\cdot\nabla\sum_{l=1}^{l_0}s_l^{(1)}\nabla\Gamma_0(\mathbf{x}-\mathbf{z}_l^{(1)})^T\mathbf{P}_l^{(1)}\mathbf{H}^{(0)}_0(\mathbf{z}_l^{(1)})\\
	    &\quad=\nu\cdot\nabla\sum_{l=1}^{l_0}s_l^{(2)}\nabla\Gamma_0(\mathbf{x}-\mathbf{z}_l^{(2)})^T\mathbf{P}_l^{(2)}\mathbf{H}^{(0)}_0(\mathbf{z}_l^{(2)}),\quad\text{on }\partial B_{R}.
	\end{aligned}
	\end{equation}
	On the other hand, it can be verified that
	\begin{equation}
		u_j(\mathbf{x}):=\sum_{l=1}^{l_0}s_l^{(j)}\nabla\Gamma_0(\mathbf{x}-\mathbf{z}_l^{(j)})\mathbf{P}_l^{(j)}\mathbf{H}^{(0)}_0(\mathbf{z}_l^{(j)}),\quad j=1,2
	\end{equation}
	are harmonic functions in $\mathbb{R}^3\setminus\overline{B_{R}}$, which decay at infinity. Using this together with \eqref{eq:4.7}, and the maximum principle of harmonic functions, one can obtain that
	\begin{equation}
		u_1(\mathbf{x})=u_2(\mathbf{x}),\quad\mathbf{x}\in\mathbb{R}^3\setminus\overline{B_{R}},
	\end{equation}
	and therefore
	\begin{equation}\label{eq:4.8}
		u_1(\mathbf{x})=u_2(\mathbf{x})\quad\mathbf{x}\in\partial B_{R}.
	\end{equation}
 By substituting \eqref{eq:3.46} and \eqref{eq:4.8} one has
\begin{equation}\label{eq:4.9}
	\sum_{n=0}^{\infty}\sum_{m=-n}^{n}\frac{\mathbf{N}^m_{n+1}(\hat{\mathbf{x}})d^{n,m}_1}{(2n+1)\|\mathbf
		{x}\|^{n+2}}=\sum_{n=0}^{\infty}\sum_{m=-n}^{n}\frac{\mathbf{N}^m_{n+1}(\hat{\mathbf{x}})d^{n,m}_2}{(2n+1)\|\mathbf
		{x}\|^{n+2}}
\end{equation}
By taking $\|\mathbf{x}\|$ sufficiently large and comparing the order of $\|\mathbf
{x}\|^{-1}$ one has \eqref{eq:4.3}.
The proof is complete.
\end{proof}

\subsection{Uniqueness in recovering a single anomaly}
We first present the uniqueness result in recovering a single anomaly.
\begin{thm}\label{thm:singleuniqueness}
	Suppose $l_0=1$ and $\sigma_1^{(1)},\sigma_1^{(2)}\neq0$. If there hold \eqref{eq:4.2} and \eqref{eq:4.2b} and $\alpha_1^{(j)}>-1$, $j=1, 2$, then $\mathbf{z}_1^{(1)}=\mathbf{z}_1^{(2)},\ \alpha_1^{(1)}=\alpha_1^{(2)}$ and $\mu_1^{(1)}=\mu_1^{(2)}$.
\end{thm}
\begin{proof}
	By the unique continuation principle, one has from \eqref{eq:4.2} that $\mathbf{H}_{s,1}=\mathbf{H}_{s,2}$ in $\mathbb{R}^3\setminus D_1^{(1)}\cup D_1^{(2)}$. First, we let $n=0$ and $l_0=1$ and note that $Y_0^0=\frac{1}{2\sqrt{\pi}}$.                      By using \eqref{eq:4.4} and \eqref{eq:4.5}, one has
	\begin{equation}\label{eq:4.10}
		\mathbf{d}_j^{0,0}=\frac{1}{2\sqrt{\pi}}s_1^{(j)}\mathbf{P}_1^{(j)}\mathbf{H}_0^{(0)}(\mathbf{z}_1^{(j)})\quad\text{and}\quad\mathbf{N}^0_1(\hat{\mathbf{x}})=\frac{1}{2\sqrt{\pi}}\hat{\mathbf{x}},\quad j=1,2.
	\end{equation}
	Hence by using \eqref{eq:4.3}, there holds
	\begin{equation}\label{eq:4.11}
		\hat{\mathbf{x}}^Ts_1^{(1)}\mathbf{P}_1^{(1)}\mathbf{H}_0^{(0)}(\mathbf{z}_1^{(1)})=\hat{\mathbf{x}}^Ts_1^{(2)}\mathbf{P}_1^{(2)}\mathbf{H}_0^{(0)}(\mathbf{z}_1^{(2)}),\quad \forall\hat{\mathbf{x}}\in\mathbb{S}^2,
	\end{equation}
	which readily implies that
	\begin{equation}\label{eq:4.12}
	s_1^{(1)}\mathbf{P}_1^{(1)}\mathbf{H}_0^{(0)}(\mathbf{z}_1^{(1)})=s_1^{(2)}\mathbf{P}_1^{(2)}\mathbf{H}_0^{(0)}(\mathbf{z}_1^{(2)}),\quad \forall\hat{\mathbf{x}}\in\mathbb{S}^2,
	\end{equation}
	or in the form with respect to $\alpha_l^{(1)}$ and $\alpha_l^{(2)}$
	\begin{equation}\label{eq:4.12b}
	\delta^{\alpha_1^{(1)}}\mathbf{P}_1^{(1)}\mathbf{H}_0^{(0)}(\mathbf{z}_1^{(1)})=\delta^{\alpha_1^{(2)}}\mathbf{P}_1^{(2)}\mathbf{H}_0^{(0)}(\mathbf{z}_1^{(2)}),\quad \forall\hat{\mathbf{x}}\in\mathbb{S}^2.
	\end{equation}
	In what follows, we define $\mathbf{C}_{s,1}^{(j)}:=s_1^{(j)}\mathbf{P}_1^{(j)}\mathbf{H}_0^{(0)}(\mathbf{z}_1^{(j)}),j=1,2$. Then $\mathbf{C}_{s,1}^{(1)}=\mathbf{C}_{s,1}^{(2)}$. We claim that $$\mathbf{H}_0^{(0)}(\mathbf{z}_1^{(j)})\neq0,j=1,2.$$
	By using Lemma 2.3 in \cite{deng2018identifying}, one has $\mathbf{H}_0^{(0)}=\nabla u_0$, where $u_0$ is a non-constant harmonic function. Hence, by the maximum principle of harmonic functions, $\nabla u_0(\mathbf{z}_1^{(j)})\neq0,j=1,2$. This together with the assumption that $\mathbf{P}_1^{(j)},j=1,2,$ are nonsingular matrices and $\delta\in\mathbb{R}_+$, one has $\mathbf{C}_{s,1}^{(j)}\neq0,j=1,2.$
	
	Now by setting $n=1$, we define
	\begin{equation}
		\mathbf{Y}_1(\hat{\mathbf{x}}):=\left(Y_1^{-1}(\hat{\mathbf{x}}),Y_1^{0}(\hat{\mathbf{x}}),Y_1^{1}(\hat{\mathbf{x}})\right)^T.
	\end{equation}
	Let $(\hat{\mathbf{x}},\vect{e}_\phi,\vect{e}_\theta)$ be the triplet of the local orthogonal unit vectors, where $\phi$ and $\theta$ depend on $\hat{\mathbf{x}}$. Note that
	\begin{equation}
		\nabla_s Y_n^m=\frac{1}{\sin\theta}\frac{\partial Y_n^m(\hat{\mathbf{x}})}{\partial \theta}\vect{e}_\phi+\frac{\partial Y_n^m(\hat{\mathbf{x}})}{\partial \theta}\vect{e}_\theta.
	\end{equation}
	By using \eqref{eq:4.3}, one has
	\begin{equation}\label{eq:4.14}
		\|\mathbf{z}_1^{(1)}\|\overline{\mathbf{Y}_1(\mathbf{z}_1^{(1)})}^T\vect{Q}(\hat{\mathbf{x}})(\hat{\mathbf{x}},\vect{e}_\phi,\vect{e}_\theta)\mathbf{C}_{s,1}^{(1)}=	\|\mathbf{z}_1^{(2)}\|\overline{\mathbf{Y}_1(\mathbf{z}_1^{(2)})}^T\vect{Q}(\hat{\mathbf{x}})(\hat{\mathbf{x}},\vect{e}_\phi,\vect{e}_\theta)\mathbf{C}_{s,1}^{(2)},\quad\hat{\mathbf{x}}\in\mathbb{S}^2,
	\end{equation}
	where $\vect{Q}(\hat{\mathbf{x}})$ is defined by
	\begin{equation}\label{eq:4.15}
		\vect{Q}(\hat{\mathbf{x}}):=\left[2\mathbf{Y}_1(\hat{\mathbf{x}}),-\frac{1}{\sin\theta}\frac{\partial \mathbf{Y}_1(\hat{\mathbf{x}})}{\partial\phi},-\frac{\partial\mathbf{Y}_1(\hat{\mathbf{x}})}{\partial\theta}\right].
	\end{equation}
	Straightforward calculations show that $\vect{Q}(\hat{\mathbf{x}})$ is nonsingular for any $\hat{\mathbf{x}}\in\mathbb{S}^2$. Since $\hat{\mathbf{x}}\in\mathbb{S}^2$ is arbitrarily given, \eqref{eq:4.14} implies that
	\begin{equation}\label{eq:4.16}
		\|\mathbf{z}_1^{(1)}\|\overline{\mathbf{Y}_1(\mathbf{z}_1^{(1)})}^T=\|\mathbf{z}_1^{(2)}\|\overline{\mathbf{Y}_1(\mathbf{z}_1^{(2)})}^T.
	\end{equation}
	Then by direct calculations, one has
	\begin{equation}
		\|\mathbf{z}_1^{(1)}\|=\|\mathbf{z}_1^{(2)}\|\quad\text{and}\quad \hat{\mathbf{z}}_1^{(1)}=\hat{\mathbf{z}}_1^{(2)}.
	\end{equation}
	
	To recover the variation parameters, by the unique continuation principle again, one sees that from \eqref{eq:4.2b} $\mathbf{H}_{1}=\mathbf{H}_{2}\in \mathbb{R}^3\setminus D_1^{(1)}\cup D_1^{(2)}$. We let $n=0$ and use Lemma 4.1 in \cite{deng2018identifying}, there holds
	\begin{equation}\label{eq:4.11b}
	\hat{\mathbf{x}}^T\mathbf{P}_1^{(1)}\mathbf{H}_0^{(0)}(\mathbf{z}_1^{(1)})=\hat{\mathbf{x}}^T\mathbf{P}_1^{(2)}\mathbf{H}_0^{(0)}(\mathbf{z}_1^{(2)}),\quad \forall\hat{\mathbf{x}}\in\mathbb{S}^2,
	\end{equation}
	which readily implies that
	\begin{equation}\label{eq:4.12deng}
	\mathbf{P}_1^{(1)}\mathbf{H}_0^{(0)}(\mathbf{z}_1^{(1)})=\mathbf{P}_1^{(2)}\mathbf{H}_0^{(0)}(\mathbf{z}_1^{(2)}),\quad \forall\hat{\mathbf{x}}\in\mathbb{S}^2,
	\end{equation}
By substituting \eqref{eq:4.12deng} in \eqref{eq:4.12} or \eqref{eq:4.12b} one has
\begin{equation}
	s_1^{(1)}=s_1^{(2)}\quad\text{or}\quad\alpha_1^{(1)}=\alpha_1^{(2)}.
\end{equation}
	We thus have $D_1^{(1)}=D_1^{(2)}$. Hence, in the sequel, we let $D_1:=D_1^{(1)}=D_1^{(2)}$. Clearly, $\mathbf{H}_{s,1}=\mathbf{H}_{s,2}$ in $\mathbb{R}\setminus\overline{D_1}$. Since $\sigma_1^{(1)},\sigma_1^{(2)}\neq0$, from \eqref{asymptotic of H_s}, one can see that
	\begin{equation}\label{eq:4.17}
		\mathbf{H}_{s,j}=\hat{\mathbf{H}}_0^{(0)}-\frac{\mu_0}{\mu_1^{(j)}-\mu_0}\nabla\mathcal{S}^0_{D_1}(\lambda_{\mu_1^{(j)}}I-(\mathcal{K}^0_{D_1})^*)^{-1}[\nu_1\cdot\hat{\mathbf{H}}_0^{(0)}]+O(\omega)\quad\text{in }\tilde{D}.
	\end{equation}
	By using the jump formula one further has that
	\begin{equation}\label{eq:4.18}
		\begin{aligned}
		&\frac{1}{\mu_1^{(1)}-\mu_0}\left(\frac{I}{2}+(\mathcal{K}^0_{D_1})^*\right)(\lambda_{\mu_1^{(1)}}I-(\mathcal{K}^0_{D_1})^*)^{-1}[\nu_1\cdot\hat{\mathbf{H}}_0^{(0)}]\vert_-\\
		&\quad=\frac{1}{\mu_1^{(2)}-\mu_0}\left(\frac{I}{2}+(\mathcal{K}^0_{D_1})^*\right)(\lambda_{\mu_1^{(2)}}I-(\mathcal{K}^0_{D_1})^*)^{-1}[\nu_1\cdot\hat{\mathbf{H}}_0^{(0)}]\vert_- \quad\text{on }\partial D_1
		\end{aligned}
	\end{equation}
	Using the fact that $\frac{I}{2}+(\mathcal{K}^0_{D_1})^*$ is invertible on $L^2(\partial D_1)$ and some elementary calculations, one sees from \eqref{eq:4.18} that
	\begin{equation}\label{eq:4.19}
		(\mu_1^{(1)}-\mu_1^{(2)})\left(\frac{I}{2}-(\mathcal{K}^0_{D_1})^{*}\right)(\lambda_{\mu_1^{(1)}}I-(\mathcal{K}^0_{D_1})^*)^{-1}[\nu_1\cdot\hat{\mathbf{H}}_0^{(0)}]=0
	\end{equation}
	Note that $\nu\cdot\hat{\mathbf{H}}_0^{(0)}\in L^2_0(\partial D_1)$, where $L^2_0(\partial D_1)$ is a subset of $L^2(\partial D_1)$ with zero average on $\partial  D_1$. Since $\nu_1\cdot\hat{\mathbf{H}}_0^{(0)}\neq0$, thus from \eqref{lem:hat H_0} that $\nu_1\cdot\hat{\mathbf{H}}_0^{(0)}\neq0$ on $\partial D_1$. Using this and the fact that $-\frac{I}{2}+(\mathcal{K}^0_{D_1})^*$ is invertible on $L^2_0(\partial D_1)$, we finally from \eqref{eq:4.19} that $\mu_1^{(1)}=\mu_1^{(2)}$. The proof is complete.
\end{proof}
\begin{rem}
We remark that for a single anomaly case, we do not need the assumption that $-1/4<\alpha_1^{(j)}< 1/3$ but only need $\alpha_1^{(j)}>-1$, $j=1, 2$. In fact, one can easily see from \eqnref{eq:asy of H_sj in term of polar} that the expansions of different orders with respect to $\delta$ can be distinguished exactly for one single anomaly.
\end{rem}

\subsection{Uniqueness in recovering multiple anomalies}
\begin{thm}\label{thm:multiuniqueness}
	 If there hold \eqref{eq:4.2} and \eqref{eq:4.2b} and there satisfy $-1/4<\alpha_l^{(j)}< 1/3$ and $3(\alpha_{l_1}^{(j)}+1)\neq 4(\alpha_{l_2}^{(j)}+1)$, $j=1,2$, $l=1,2,\dots,l_0$ for any $l_1, l_2\in \{1, 2, \ldots, l_0\}$, then $\mathbf{z}_l^{(1)}=\mathbf{z}_l^{(2)}$, $\alpha_l^{(1)}=\alpha_l^{(2)}$ and $\mu_l^{(1)}=\mu_l^{(2)}$, $l=1,2,\dots,l_0$.
\end{thm}
\begin{proof}
	With our earlier preparatiosn, the proof follows from a similar argument to that of Theorem 7.8 in \cite{ammari2007polarization} and a similar idea to that of Theorem 4.2 in \cite{deng2018identifying}. In the following, we only sketch it. Using the formula \eqref{eq:4.7} and analysis similar to the proof of Lemma \ref{lem:sumND}, one can show that
	\begin{equation}
		\sum_{l=1}^{l_0}\left(s_l^{(1)}\nabla\Gamma_0(\mathbf{x}0-\mathbf{z}_l^{(1)})^T\mathbf{P}_l^{(1)}\mathbf{H}_0^{(0)}(\mathbf{z}_l^{(1)})-s_l^{(2)}\nabla\Gamma_0(\mathbf{x}-\mathbf{z}_l^{(2)})^T\mathbf{P}_l^{(2)}\mathbf{H}_0^{(0)}(\mathbf{z}_l^{(2)})\right)=0\quad\text{in }\mathbb{R}^3\setminus\overline{\Sigma}.
	\end{equation}
By straightforward calculations, one can further show that
\begin{equation}\label{eq:4.21}
	\begin{aligned}
	F^s(\mathbf{x})&:=\sum_{l=1}^{l_0}\Big[\left(\nabla\Gamma_0(\mathbf{x}-\mathbf{z}_l^{(1)})-\nabla\Gamma_0(\mathbf{x}-\mathbf{z}_l^{(2)})\right)^Ts_l^{(1)}\mathbf{P}_l^{(1)}\mathbf{H}_0^{(0)}(\mathbf{z}_l^{(1)})\\
	&\quad-\nabla\Gamma_0(\mathbf{x}-\mathbf{z}_l^{(2)})^T\left(s_l^{(2)}\mathbf{P}_l^{(2)}\mathbf{H}_0^{(0)}(\mathbf{z}_l^{(2)})-s_l^{(1)}\mathbf{P}_l^{(1)}\mathbf{H}_0^{(0)}(\mathbf{z}_l^{(1)})\right)\Big]\\
	&=\sum_{l=1}^{l_0}\Big[\left(\nabla^2\Gamma_0(\mathbf{x}-\mathbf{z}_l')(\mathbf{z}_l^{(1)}-\mathbf{z}_l^{(2)})\right)^Ts_l^{(1)}\mathbf{P}_l^{(1)}\mathbf{H}_0^{(0)}(\mathbf{z}_l^{(1)})\\
	&\quad-\nabla\Gamma_0(\mathbf{x}-\mathbf{z}_l^{(2)})^T\left(s_l^{(2)}\mathbf{P}_l^{(2)}\mathbf{H}_0^{(0)}(\mathbf{z}_l^{(2)})-s_l^{(1)}\mathbf{P}_l^{(1)}\mathbf{H}_0^{(0)}(\mathbf{z}_l^{(1)})\right)\Big]=0\quad\text{in }\mathbb{R}^3\setminus\overline{\Sigma},
	\end{aligned}
\end{equation}
where $\mathbf{z}_l'=\mathbf{z}_l^{(1)}+t'\mathbf{z}_l^{(2)}$ with $t'\in(0,1)$. Note that $F^s(\mathbf{x})$ defined in \eqref{eq:4.21} is also harmonic in $\mathbb{R}^3\setminus\bigcup_{l=1}^{l_0}(\mathbf{z}_l^{(1)}\cup\mathbf{z}_l^{(2)})$. By using the analytic continuation of harmonic functions, one thus has that $F^s(\mathbf{x})\equiv0\in\mathbb{R}^3$. Define $F^s:=F_1^s+F_2^s$, where
\begin{equation}
	F_1^s(\mathbf{x}):=\sum_{l=1}^{l_0}\Big[\left(\nabla^2\Gamma_0(\mathbf{x}-\mathbf{z}_l')(\mathbf{z}_l^{(1)}-\mathbf{z}_l^{(2)})\right)^Ts_l^{(1)}\mathbf{P}_l^{(1)}\mathbf{H}_0^{(0)}(\mathbf{z}_l^{(1)})\Big],
\end{equation}
and
\begin{equation}\label{eq:F^s_2}
	F^s_2(\mathbf{x}):=-\sum_{l=1}^{l_0}\Big[\nabla\Gamma_0(\mathbf{x}-\mathbf{z}_l^{(2)})^T\left(s_l^{(2)}\mathbf{P}_l^{(2)}\mathbf{H}_0^{(0)}(\mathbf{z}_l^{(2)})-s_l^{(1)}\mathbf{P}_l^{(1)}\mathbf{H}_0^{(0)}(\mathbf{z}_l^{(1)})\right)\Big].
\end{equation}
Then by comparing the types of poles of $F_1^s$ and $F_2^s$, one immediately finds that $F_1^s=0$ and $F_2^s=0$  in $\mathbb{R}^3$. Since $\mathbf{H}_0^{(0)}(\mathbf{x})$ does not vanish for $\mathbf{x}\in \mathbb{R}^3\setminus\overline{\Sigma}$, then one has $s_l^{(1)}\mathbf{P}_l^{(1)}\mathbf{H}_0^{(0)}(\mathbf{z}_l^{(1)})\neq0$. Hence one sees that
\begin{equation}
	\mathbf{z}_l^{(1)}-\mathbf{z}_l^{(2)}=0,\quad l=1,2,\dots l_0.
\end{equation}
By using \eqref{eq:F^s_2} and $F^s_2=0$, one has
\begin{equation}
	s_l^{(1)}\mathbf{P}_l^{(1)}\mathbf{H}_0^{(0)}(\mathbf{z}_l^{(1)})=s_l^{(2)}\mathbf{P}_l^{(2)}\mathbf{H}_0^{(0)}(\mathbf{z}_l^{(2)}),\quad l=1,2,\dots,l_0.
\end{equation}
 Again, by similar argument to above and using the proof of Theorem 4.2 in \cite{deng2018identifying},
 \begin{equation}
\mathbf{P}_l^{(1)}\mathbf{H}_0^{(0)}(\mathbf{z}_l^{(1)})=\mathbf{P}_l^{(2)}\mathbf{H}_0^{(0)}(\mathbf{z}_l^{(2)}),\quad l=1,2,\dots,l_0.
 \end{equation}
 Hence we have
 \begin{equation}
 	s_l^{(1)}=s_l^{(2)}\quad\text{or}\quad\alpha_l^{(1)}=\alpha_l^{(2)},\quad l=1,2,\dots,l_0.
 \end{equation}
Thus by unique continuation one further has
\begin{equation}\label{eq:newadd01}
\mathbf{H}_{s,1}=\mathbf{H}_{s,2}, \quad \mbox{in} \quad \mathbb{R}^3\setminus\overline{\tilde D}.
\end{equation}
To prove that $\mu_l^{(1)}=\mu_l^{(2)}$, one only needs to apply the following transmission conditions:
$$
\nu\cdot(\mu_0\mathbf{H}_{s,j})\Big|_+=\nu\cdot(\mu_l^{(j)}\mathbf{H}_{s,j})\Big|_-, \quad \mbox{on} \quad \partial D_l,
$$
for $j=1,2$ and use \eqnref{eq:newadd01}. The proof then follows similarly to \cite{cdengMHD}.
\end{proof}

\section{Concluding Remark}\label{concluding}
In this paper, we consider the identification of magnetic anomalies with varing parameters beneath the Earth by using geomagnetic monitoring. The work is an extension of the existing results in \cite{deng2018identifying} on identification of magnetized anomalies (size and material parameters). In this paper, the magnetic anomalies are supposed to grow or shrink when time elapses and our main aim is to identify the anomalies (size, material parameters, varying scale) by observation of the geomagnetic variation--the so-called secular variation. In order to avoid the use of the background geomagnetic field, which is not assumed to be known from a practical viewpoint, new technique strategy is developed in this paper. Our results provide a rigorous mathematical theory to some existing applications in the geophysics community.

\section*{Acknowledgment}
The work of Y. Deng was supported by NSF grant of China No. 11601528, NSF grant of Hunan No. 2017JJ3432 and No. 2018JJ3622, PSCF and RFEB of Hunan No. 18YBQ077 and No. 18B337,  Innovation-Driven Project of Central South University, No. 2018CX041. The work of H. Liu was supported by the FRG and startup grants from Hong Kong Baptist University, Hong Kong RGC General Research Funds, 12302017 and 12301218.

\end{document}